\documentclass[11pt,oneside,reqno]{amsart}

\usepackage{amssymb}
\usepackage{amsmath, amsfonts}
\usepackage{amsthm}
\usepackage{epsfig}
\usepackage{graphicx}
\usepackage{graphics}
\usepackage{float}
\usepackage{subfigure}
\usepackage{multirow}
\usepackage{color}
\usepackage{lineno}
\usepackage{fullpage}
\usepackage[normalem]{ulem} 
\usepackage{xspace}
\usepackage{wrapfig}
\usepackage{amsmath,amsfonts,amssymb,amscd,amsthm,amsbsy,epsf}
\usepackage{euscript}

\usepackage[colorlinks=true,citecolor=blue]{hyperref}

\numberwithin{equation}{section}
\newtheorem{theorem}{Theorem}
\newtheorem{meta-thm}[theorem]{Meta-Theorem}
\newtheorem{lemma}[theorem]{Lemma}
\newtheorem{cor}[theorem]{Corollary}
\newtheorem{proposition}[theorem]{Proposition}
\newtheorem{algorithm}[theorem]{Algorithm}
\newtheorem{remark}[theorem]{Remark}
\newtheorem{definition}[theorem]{Definition}

\newcommand\elem{ \end{lemma} }
\newcommand\lem[1]{ \begin{lemma}\label{#1}}

\newcommand{\noaverage}[1]{({#1})^0}
\newcommand\normm[2]{\left\|{#1}\right\|_{#2}}
\newcommand\norm[3]{\left\|{#1} \right\|_{#2,#3}}
\newcommand\expan[2]{#1^{[\leq #2]}}
 
\newcommand\expanin[3]{#1^{(#2,#3]}}
\newcommand\ord[1]{\mathcal{O}\left(|\varepsilon |^{#1}\right)}

\newcommand{\Aa}[2]{\mathcal{A}_{#1,#2}}

\newcommand\beq[1] { \begin{equation}\label{#1} }
\newcommand{\eeq}{ \end{equation} }

\newcommand\beqa[1]{ \begin{eqnarray} \label{#1}}
\newcommand{\eeqa}{ \end{eqnarray} }
\newcommand{\beqano}{ \begin{eqnarray*} }
\newcommand{\eeqano}{ \end{eqnarray*} }

\newcommand{\balino}{ \begin{align*} }
\newcommand{\ealino}{ \end{align*} }

\definecolor{indiagreen}{rgb}{0.07, 0.53, 0.03}

\def\e{\varepsilon}
\def\bmid{\mathop{\,\big|\,}}
\def\dist{\operatorname{dist}}
\def\domain{\operatorname{Domain}}

\def\Id{\operatorname{Id}}
\def\Im{\operatorname{Im}}

\def\det{\operatorname{det}}

\def\A{{\mathcal A}}
\def\B{{\mathcal B}}
\def\C{{\mathcal C}}
\def\D{{\mathcal D}}
\def\E{{\mathcal E}}

\def\G{{\mathcal G}}

\def\M{{\mathcal M}}
\def\N{{\mathcal N}}
\def\R{{\mathcal R}}

\def\O{{\mathcal O}}
\def\T{{\mathcal T}}
\def\complex{{\mathbb C}}
\def\integer{{\mathbb Z}}
\def\nat{{\mathbb N}}
\def\real{{\mathbb R}}	
\def\torus{{\mathbb T}}
\newcommand{\dpy}{\displaystyle}

\def\r{{\rho}}
\def\g{{\gamma}}
\def\l{{\lambda}}
\def\d{{\delta}}
\def\t{{\theta}}	

\newcommand{\Addresses}{{
\bigskip
\footnotesize

A.P.~Bustamante, \textsc{School of Mathematics, Georgia Institute of Technology}\par\nopagebreak \textit{E-mail address},\texttt{apb7@math.gatech.edu}

\medskip

R.~de la Llave, \textsc{School of Mathematics, Georgia Institute of Technology}\par\nopagebreak \textit{E-mail address},\texttt{rafael.delallave@math.gatech.edu}

}}

\begin{document}

\title{Gevrey estimates for asymptotic expansions of tori in weakly dissipative systems}

\author{ Adri\'an P. Bustamante \and  Rafael de la Llave } 

\thanks{Both authors have been supported by NSF grant DMS 1800241}



\begin{abstract}

We consider a singular perturbation for a family of analytic symplectic maps of the annulus possessing a KAM torus. The perturbation introduces dissipation and  contains an adjustable parameter. By choosing the adjustable parameter, one can ensure that the torus persists under perturbation. Such models are common in celestial mechanics.  In field theory, the adjustable parameter is called \emph{the counterterm} and in celestial mechanics, the \emph{drift}. It is known that there are formal expansions in powers
of the perturbation both for the quasi-periodic solution and the counterterm.

We prove that the asymptotic expansions for the quasiperiodic solutions and the counterterm satisfy Gevrey estimates. That is, the $n$-th term of the expansion is bounded by a power of $n!$. The Gevrey class (the power of $n!$) depends only on the Diophantine condition of the frequency and the order of the friction coefficient in powers of the perturbative parameter.

The method of proof we introduce may be of interest beyond the problem considered here. We consider a modified Newton method in a space of power expansions. As it is custumary in KAM
theory, each step of the method is estimated in a smaller domain. In contrast with the KAM results, the domains where we control the Newton  method  shrink very fast and the Newton method does not prove that the solutions are analytic. On the other hand, by examining carefully the
process, we can obtain estimates on the coefficients of the expansions and conclude
the series are Gevrey. 

\end{abstract}

\maketitle

\section{Introduction}

Hamiltonian systems with small dissipation appear as models of many problems of physical interest.  Notably, dissipation is a small effect
in astrodynamics of planets and satellites \cite{MilaniNF87,Celletti13} 
\footnote{A problem in astrodynamics which motivate us is  the 
\emph{spin orbit problem} describing approximately the motion of an 
oblate planet, subject to tidal friction, in a Keplerian orbit \cite{Celletti91}}. 
In the design of many mechanical devices, 
eliminating friction is a design goal which is never 
completely accomplished. Hamiltonian systems with friction also 
 appear as
Euler-Lagrange equations of discounted functionals which are natural
in finance and in the receding horizon problem 
in control theory. In such a case the limit
of zero discount (equivalent to 
the limit of zero friction)  is of interest.  See
\cite{Bensoussan88,MeadowsHER95,IturriagaS11, DaviniFIZ16}  
for different studies of the zero dissipation limit in 
calculus of variations 
and in control.

Since the friction is small, it is natural to 
try to study such systems using perturbation theory. 
Nevertheless, adding a small friction is 
a very singular perturbation, and  periodic/quasi-periodic orbits may 
disappear for arbitrarily small values or the perturbation. 
In contrast with Hamiltonian systems that often have 
sets of quasi-periodic orbits of positive measure (KAM theorem), 
for  dissipative forced systems, there are few periodic or 
quasi-periodic orbits. These quasi-periodic orbits are 
known to persist only if one can adjust parameters in 
the system \cite{Moser67,BroerHS96,Sevryuk99}. 
As discussed very clearly in \cite{Moser73}, the number of 
parameters needed is affected by the geometric properties of the systems considered. 

In recent times, for some 
particular types of dissipative systems -- the conformally symplectic 
systems, see Definition \ref{def:conformallysymplectic} -- 
there is a very systematic KAM theory \cite{Cal-Cel-Lla-13}
based on geometric arguments. The examples mentioned above (Hamiltonian 
systems with friction proportional to the momentum and Euler-Lagrange 
equations of exponetially discounted variational principles) 
are conformally symplectic.  This theory, once we fix a frequency, 
predicts the changes of parameters and the changes in the solutions 
needed to obtain a quasi-periodic solution of the prescribed frequency. 

The goal of this paper is to study the singular perturbation theories
in which the perturbation introduces dissipation. 

There are several studies of the singular perturbation 
theories in dissipation which are particularly relevant for us:
The paper \cite{Cal-Cel-Lla-16}  shows that if one fixes 
a Diophantine frequency  $\omega$ (see Definition~\ref{diophantine}), considers a Hamiltonian system -- not necessarily 
integrable -- with a quasi-periodic solution of frequency $\omega$, and introduces a conformally symplectic perturbation (see Definition~\ref{def:conformallysymplectic}), then
there is a (unique under a natural normalization) 
formal power series expansion  for the quasi-periodic 
solution of frequency $\omega$ and for the drift parameter. These series are very similar 
to the Lindstedt series of classical mechanics. 
 The paper \cite{Cal-Cel-Lla-16} also showed
that the formal Lindstedt series is  the asymptotic expansion of 
a true solution defined in a complex domain of parameters that does 
not include any ball around zero (giving an indication that the power
series may be divergent).  The paper \cite{Bus-Cal-19} studied numericaly 
these Lindstedt series in a concrete example 
and the possible domain of analyticity of the function (using Pad\'e as well 
as non-perturbative methods). 
The numerical studies in \cite{Bus-Cal-19} lead to the remarkable conjecture that, in the cases 
examined,  the formal power series giving 
the quasiperiodic solution and the forcing 
are  Gevrey (see Definition~\ref{gevrey_def}). 

In this paper, for some class of maps (we require that 
the system is conformaly symplectic and that the non-linearity 
is a trig. polynomial) we show that the conjecture in \cite{Bus-Cal-19} 
is true and that the series obtained are indeed Gevrey. The 
Gevrey class can be bounded depending only on the Diophantine 
condition of the frequency $\omega$ (and the order of the friction in 
the dissipation).  See Theorem~\ref{main_theo_dsm}. 

The method of proof we introduce may be of interest beyond the problem considered here and we hope that there are other applications. 
We consider a Newton method in the space of power expansions. 
As in KAM theory, each step of the quadratically convergent method is estimated in a domain smaller than the domain of the previous steps. In contrast with KAM theory, 
the domains where we control the results shrink very fast to a point, so 
that, at the end we do not obtain any analytic function. On the 
other hand, by examining carefully the process, we can obtain estimates
on the coefficients of the expansions. 

Our hypothesis that the  non-linearity is a trigonometric polynomial ensures that the coefficients of 
order $N$ do not change after $\log_2(N)$ steps of the Newton method, so
that one can use Cauchy estimates in the domain that is under control after 
 $\log_2(N)$ steps to obtain estimates on the $N$th coefficient. 

We hope that the hypothesis that the non-linearity is a trigonometric polynomial can be removed
at the price of estimating the change of the coefficients in subsequent 
iterations, but a proof would require a new 
set of estimates that -- if indeed possible -- would lengthen the  
exposition and obscure the main ideas. 

The Newton method acting on power series is patterned after the Newton 
method used in \cite{Cal-Cel-Lla-13}. This Newton method takes 
advantage of remarkable cancellations related to the geometry
and introduces the corrections to the torus additively (rather 
than making changes of variables). The fact that 
the Newton method in \cite{Cal-Cel-Lla-13} does 
not involve changes of variables  makes it possible 
to lift it to  formal power series. We will present full details
later. 

For simplicity in the treatment, we will deal with  maps since 
the geometric arguments are simpler. The same arguments apply 
for differential equations, but they are more elaborate. Besides 
adapting the proof of maps to the case of ODE's, one 
can  deduce  rigorously the results for differential equations from 
the results for maps by taking time-$T$ maps. Note that in this case, 
the fact that the non-linearity in the time-$T$ map is a trig. polynomial
is difficult to express in terms of the original ODE. This 
is another reason why we would like eventually to get rid of 
that hypothesis. 


\subsection{A preview of the main result}

A model to keep in mind is the so-called dissipative standard map $f_{\e,\mu_\e}:\torus\times\real \longrightarrow \torus\times\real$ given by
\begin{equation}\label{diss-est-map}
    f_{\e,\mu_\e}(x,y)= (x+\l(\e)y + \mu_\e -\e V'(x) , \l(\e)y + \mu_\e -\e V'(x)) 
\end{equation}
In  \eqref{diss-est-map}, the physical meaning of  $\l(\e)=1-\e^\alpha$, $\alpha\in\nat$, is dissipation and $\mu_\e$, called the \emph{drift} parameter, has the physical meaning of a forcing. Our assumption on the non-linearity amounts to $V$ being a trigonometric polynomial. The model \eqref{diss-est-map} is indeed conformally symplectic in the sense of Definition \ref{def:conformallysymplectic} (see below). The map~\eqref{diss-est-map} is the model that was used in the numerical experiments in \cite{Bus-Cal-19}. 


Note that for $\e = 0$, the map \eqref{diss-est-map} is integrable.
The integrability of the map at $\e = 0$ does not play any role 
in the theoretical results in \cite{Cal-Cel-Lla-16},  the only assumption 
needed in \cite{Cal-Cel-Lla-16} is that map for $\e= 0$ is symplectic and 
has as an invariant torus. For the numerical study in 
\cite{Bus-Cal-19}, the fact that the map for $\e= 0$ is integrable 
leads to much more efficient algorithms.  In this paper, we will not use 
explicitly the integrability for $\e=0$, but this seems to be 
the only case where it is possible to verify the assumption 
on the nonlinearity being a trig polynomial (yet another reason 
to try to get rid of that hypothesis). 

The main result of this paper, Theorem \ref{main_theo_dsm}, establishes the Gevrey character of the formal power series expansions for the drift parameter $\mu_\e$ and for the quasi-periodic orbit of frequency $\omega$ of the map \eqref{diss-est-map}. The rigorous formulation of the main Theorem is given in Section \ref{sec_main_res}, the statements of the main results can be better understood after some preliminary definitions and remarks are given (see Section \ref{Prelims}). Here we give an informal statement of our main result: Given a Diophantine frequency $\omega$, the coefficients of the formal power series expansions $\sum K_n\e^n$ and $\sum \mu_n\e^n$  for the quasi-periodic orbit and the drift parameter, respectively, satisfy the following Gevrey estimates $$\|K_n\| \leq CR^nn^{(2\tau/\alpha)n}  \qquad |\mu_n|\leq CR^n n^{(2\tau/\alpha)n} $$  where $\tau$ depends on the Diophantine type of $\omega$ (see Definition \ref{diophantine}) and $\alpha$ is the order of the dissipation $\l(\e) = 1 - \e^\alpha$.

The model~\eqref{diss-est-map} can be thought as a numerical time step -- using a Verlet-like method -- of 
the spin-orbit problem 
\begin{equation}\label{spin-orbit1} 
\begin{split} 
&\dot x = y\\
&\dot y = - \mu y + \l  + V'(x) 
\end{split} 
\end{equation}

\subsection{Organization of the paper}

The paper is organized as follows. In Section \ref{Prelims} we collect some standard definitions and we also define the function spaces in which the iterative procedure takes place. Also, in the same section we present some geometric identities which allow us to solve the linearized equations of the modified Newton method. In Section \ref{sec_main_res} we state Theorem \ref{main_theo_dsm} and Lemma \ref{main_theo}, which are the main results of the paper and establish the  Gevrey character of the perturbative expansions of the quasi periodic orbits.

The proof of Theorem \ref{main_theo_dsm} is based on a quasi Newton method. In Section \ref{sec_new_step} we formulate the iterative step of this Newton method, while in Section \ref{sec_nolin_est} we provide estimates for the corrections and the new error at one  step of the method. Finally, in Section \ref{sec_proof_main}, using a KAM like argument, we give estimates for any step of the Newton like procedure and, with them,  a proof of Lemma \ref{main_theo} is given establishing the Gevrey character of the perturbative expansions.

\section{Preliminaries } \label{Prelims}

In this section we introduce  the 
notations, collect some standard definitions 
including the 
Banach spaces and their norms that enter in 
this  paper.  This section should be used as a reference.

\subsection{Symplectic properties} 

Let $\M = \torus^d\times B $, $B\subseteq \real^d$; endowed with an
exact symplectic form $\Omega$. 
Note that the manifold $\M$ is Euclidean (i.e. the tangent bundle is 
trivial) and we can compare vectors in different tangent spaces. 
This is crucial in KAM theory. 

We denote by $J$ the matrix
associated to the symplectic form $\Omega$, i.e., in coordinates we
have $\Omega_x(u,v) = (u,J(x)v)$ 
where $(\cdot, \cdot)$ denotes the inner product for any $u,v\in T_x\M$.
Note that $J$ depends on the choice of the inner product.

\begin{definition}\label{def:conformallysymplectic} 
We say that a diffeomorphism defined on an symplectic 
manifold $(\M, \Omega)$ is  conformally symplectic
when  $$  f^* \Omega = \lambda \Omega $$  for a number $\lambda$, where 
$f^*$ denotes the standard pull back on forms. 
\end{definition} 

The map \eqref{diss-est-map} is conformally symplectic with the conformal factor $\lambda(\e) = 1-\e^\alpha$ and the standard symplectic form $\Omega = dx\wedge dy$ on the cylinder $\torus \times \real$.

\subsection{Banach spaces of analytic functions} 

\subsubsection{Analytic functions on the torus} 

Given $\r>0$ we define the complex extension of the $d$-dimensional
torus as $$\torus_\rho^d=\left\{z\in \complex^d/\integer^d \,|\,
\operatorname{Re}(z_j)\in \torus,\,|\Im(z_j)|\leq \rho\right\} $$ and
denote $\A_\rho$ as the vector space of analytic functions defined
$\operatorname{int}(\torus_\rho^d)$ which can be extended continuously
to the boundary of $\torus_\rho^d$. $\A_\rho$ is endowed with the
norm $$\normm g \r =\sup_{\theta\in \torus^d_\rho} |g(\theta)|$$ which makes it into a Banach space.

For vector valued functions, $g=(g_1,g_2,..., g_d)$, we define the
norm $$\normm{g}{\r} = \sqrt{\normm{g_1}{\r}^2 +\normm{g_2}{\r}^2+...+
  \normm{g_d}{\r}^2}$$ and for $n_1\times n_2$ matrix valued functions,
G, we define $$\normm{G}{\r} = \sup_{v\in \real^{n_2},|v|=1}
\sqrt{\sum_{i=1}^{n_1} \left(
  \sum_{j=1}^{n_2}\normm{G_{ij}}{\r}v_j\right)^2 }. $$

We will also need to work with functions of 
two variables. Denoting $B_\g(0)\subseteq \complex$ the open ball
with center zero and radius $\g$ in the complex plane, define
$$\A_{\r,\g}=\left\{K:{B_\g(0)}\rightarrow \A_\rho \bmid K\mbox{ is analytic in }B_\g(0) \mbox{ and can be extended continuously to } \overline{B_\g(0)} \right\} $$ endowed with the norm $$\norm K\r\g:=\sup_{|\e|\leq\g}\normm{K(\e)}{\r}.$$ 
It is well known that with the norms $\norm{\cdot }{\r }{\g}$ and $\normm{\cdot }{\r}$ the spaces $\A_{\r, \g}$ and $\A_\r$ are Banach algebras.

To discuss analyticity properties, we will need to deal 
with complex values of all the arguments. For phyical applications, 
we need mainly real variables. Hence, it will be important that the 
functions we consider have the property that they yield real values 
for real arguments.   The functions that satify this property (real valued 
for real arguments) is 
a closed (real) subspace of the above Banach spaces.  
All the constructions we use have the property 
that when applied to real valued functions, they produce real valued functions.



Note that we can think of functions $\A_{\r,\g}$ as
analytic functions on $B_\g(0)$ taking values on a space of 
analytic functions of the torus. This point of view is 
consistent with the interpretation that we are considering families of 
problems and we are seeking families of solutions. 

For typographical reasons from now on we will use the following notation.
Given $K\in\Aa\r\g$ we denote $K_\e(\t)=K(\t,\e):=(K(\e))(\t)$.

\begin{definition}
Let $\B$ a Banach space. Given an analytic function $g: B_\g(0)\subseteq\complex\longrightarrow \B$, and $n\geq 0$, we say $g(\e)\sim \ord n$ if and only if there exists $C>0$ such that $$\|g(\e)\|\leq C|\e|^n$$for $\e$ small enough. Equivalently, $g(\e)\sim \ord n$ if and only if $g(\e)=\sum_{k=n}^\infty g_k\e^k$ for $\e$ small enough and $g_k \in \B$.  
\end{definition}

\subsubsection{Cauchy estimates.}

We recall the classical Cauchy inequalities, see \cite{Zygmund65}.

\begin{lemma} For any $0<\d\leq \r$ and for any function $f\in \A_\r$ we have $$\normm {D^nf}{\r-\d}\leq C\d^{-n}\normm f\r,$$ where $D^n$ denotes the n-th derivative and $$|\hat{f}_k|\leq e^{-2\pi|k|\r}\normm f\r $$ where $|k|=|k_1|+|k_2|+\dots +|k_n|$ and $\hat{f}$ denotes the Fourier coefficient of $f$ with index $k$.
\end{lemma}

As mentioned above we will be working with functions depending upon two variables. The following are Cauchy inequalities in the second variable, $\e$. 

\lem {cauchy_epsilon}

For any $0<r\leq\g$ and any function $f\in \A_{\r,\g}$ such that $f_\e(\t)=\sum_{n=0}^\infty f_n(\t)\e^n$ we have $$\normm{f_n}\r\leq \frac{1}{r^n}\norm f\r r .$$
\elem
\begin{proof} By Cauchy integral formula
$$f_n(\t)=\left.\frac{1}{n!}\frac{d^n}{d\e^n}f(\theta,\e)\right|_{\e=0}
= \frac{1}{2\pi i}\int_{|\xi|=r}\frac{f(\t,\xi)}{\xi^{n+1}}d\xi
=\frac{1}{2\pi r^n}\int_0^{2\pi}\frac{f(\t,re^{i\phi})}{e^{in\phi}}d\phi,$$ thus, $\dpy{|f_n(\theta)|\leq\frac{1}{r^n}\sup_{|\e|\leq r}|f(\t,\e)|}$ and $\dpy{\normm{f_n}\r\leq\frac{1}{r^n}\norm f\r r}$.
\end{proof}

\begin{cor}\label{lem_cauchy_est_ser}
Assume that $\Delta\in \Aa\r\g$ is such that $\Delta_\e =\sum_{n=N+1}^\infty \Delta_n\e^n$. Let $a,b\in\nat$ such that $N\leq a < b\leq \infty$ and denote $\expanin {\Delta_\e} a b = \sum_{n=a+1}^b \Delta_n\e^n$. Then, for all $0<r<1$ we have $$\norm{\expanin \Delta ab}{\r}{r\g} \leq \frac{r^{a+1}}{1-r}\norm \Delta\r\g. $$
\end{cor}

\begin{remark}
Note that the estimate in Corollary \ref{lem_cauchy_est_ser} only depends on $a$, associated with  the order of the first term in the expansion of $\expanin \Delta ab$.
\end{remark}

\subsection{Formal power series} 
\subsubsection{General definitions} 
Formal power series expansions are 
just expressions of the form 
\[
\sum_n a_n \e^n 
\]
where $a_n$ belong to a Banach space, sometimes $a_n$ are 
just scalars. 

Formal power series are not meant to converge nor to represent a function. 
They can, however be added, multiplied (using the Cauchy formula for 
product; note that for a fixed degree, computing the coefficients involves only a 
finite sum) or substituted one into another. 

One can form equations among formal power series. The meaning is, 
of course, that the coefficients on each side should be the same. 
This is extremely useful in many areas of mathematics, notably combinatorics. 
See \cite{Cartan95}, \cite{Cos-09} for more details on formal power series. 

Many perturbation expansions in Physics or 
in applied mathematics are based precisely into formulating 
the solutions of the equations of motion as formal power series and requiring that the 
equations of motion  are satisfied in the sense of power series.  Notably, 
the Lindstedt series were in standard use in astronomy even 
if they were only shown to converge for some frequencies in \cite{Moser67}. 
\subsubsection{Asymptotic expansions} 

For formal power series, a notion weaker that convergence of the series 
to a function  is that the series is
asymptotic to a function. 

\begin{definition} \label{asymptotic} 
We say that a formal power series $\sum a_n \e^n$ with coefficients 
$a_n$ in a Banach space $X$, is an 
asymptotic expansion to a function $\phi: \mathcal{D} \rightarrow X$ 
when for all $N \in \integer$, there exists $C_N$  such  that for 
all $\rho < \rho_0 $
\[
\sup_{\e \in \mathcal{D}, |\e| \le \rho} \left\| \sum_{n=0}^N a_n \e^n - \phi(\e) \right\|
\le C_{ N} \rho^{N+1} 
\]
\end{definition} 

If the domain $\mathcal{D}$ does not include any ball centered 
at zero, even if the function $\phi$ is analytic and bounded on $\mathcal{D}$, this 
does not imply that the series converges.

Given a function $\phi$, the associated expansions may be non unique. 
The Cauchy example 
\begin{equation}\label{cauchy-ex}
\phi(\e) = \exp( - \e^{-2}) 
\end{equation} 
has an identically zero 
asymptotic expansion on a domain
\begin{equation}\label{Ddomain} 
\mathcal{D}_\delta = \{ \e  :\, | \text{Arg}(\e)| < \delta \} 
\end{equation}
when $\delta < \pi$. 

Note that the definition of asymptotic involves the domain $\mathcal{D}$. 
A series may be asymptotic to a function in a domain but not in a  
larger domain. For example the zero series 
is asymptotic to the  the Cauchy example~\ref{cauchy-ex} in the domains
$\mathcal{D}_\delta$ as in \eqref{Ddomain} when $\delta < \pi$,
but not when $\delta > \pi$. 

\subsubsection{Gevrey formal expansions}

Given a formal power series, even if it diverges, it is interesting to study how
fast the coefficients grow. The following definition captures some 
speed of growth that is weaker than convergence, but which nevertheless 
appears naturally in many applied problems.

\begin{definition}\label{gevrey_def} Let $\beta,\r > 0$. We say that a  power series expansion $f=\sum_{n=0}^\infty f_n(\t)\e^n$, with $f_n\in \A_{\r}$, belongs to a Gevrey class $(\beta,\r)$ if and only if there exist constants $C\geq 0$ , $R \geq 0$, and $n_0\in\nat$ such that 
\begin{equation} \label{gev-def-f}
\normm {f_n}{\r} \leq CR^n n^{\beta n}\quad\mbox{for } n\geq n_0,
\end{equation}
and we denote $f\in \G_{\r}^\beta$.

Similarly, we say that a power series expansion  $\mu= \sum_{n=0}^\infty \mu_n\e^n$, with $\mu_n\in \complex^d$, belongs to a Gevrey class $\beta$ if and only if there exist constants $C\geq 0$ , $R \geq 0$, and $n_0\in\nat$ such that \begin{equation} 
\left| \mu_n\right| \leq CR^n n^{\beta n}\quad\mbox{for } n\geq n_0,
\end{equation}
and we denote $\mu \in \G^\beta$. 
\end{definition}

\begin{remark}
It is well known that \eqref{gev-def-f} in Definition \ref{gevrey_def} is equivalent to the inequality $$\normm{f_n}{\r} \leq CR^n(n!)^\beta\quad \mbox{for }n\geq n_0 $$ which, in turn, implies the series $\sum_{n=0}^\infty\frac{f_n(\t)}{(n!)^\beta} \e^n$ converges
in $\A_\r$ with positive radius of convergence.

This remark makes a connection with the theory of Borel summability. 
If a series is Gevrey, under some extra conditions, the Borel 
transform produces a function that is analytic in a sector and the series 
is asymptotic to this function. See \cite{Cos-Gal-Gen-Giu-07}, \cite{Cos-09}.
\end{remark}

\begin{remark} 
 The class of functions 
that around each point have expansions satisfying 
 Definition~\ref{gevrey_def} has received a lot of interest recently since those 
functions are related to many deep theorems  of Dynamical Systems
(KAM, Nekhoroshev). Similar theories (e.g. hypoellipticity)
also admit Gevrey classes as natural regularity.

This paper goes in a different direction. Even if we 
start with an analytic problem -- indeed polynomial! --
several objects of interest are only Gevrey.   
The phemenon that Analytic problems 
have only Gevrey solutions has appeared in other contexts in dynamics, 
notably in the study of singular perturbations \cite{CanalisRSS00}, 
the regularity of attractors and fast-slow systems
 \cite{FoiasT89,Canalis91, Baesens95}. 
Closer to us, in dependence on parameters of 
solutions of non-linear problems,  \cite{Sauzin92,Lin92}, dependence 
of KAM tori in the frequency \cite{Popov00}, 
or in the theory 
of parabolic manifolds \cite{BaldomaH08, BaldomaFM17}.
\end{remark}

\subsubsection{A property from number theory} 
In KAM theory, some number theoretical properties of 
frequencies 
play an important role. We will use the standard:

\begin{definition}\label{diophantine} 
For $\nu, \tau > 0 $, 
we say $\omega\in\real^d$ is Diophantine of type $(\nu,\tau)$ if 
$\dpy{\left|e^{2\pi ik\cdot\omega}-1\right|\geq\nu|k|^{-\tau}}$.

We denote $\omega\in\D(\nu,\tau)$.
\end{definition}

\subsection{Quasi-periodic orbits}
\label{qp-functions}

A quasi-periodic sequence $\{x_n\}_{n \in \integer} $
 of frequency $\omega\in \real^d $ in a Euclidean manifold is a 
sequence which can be expressed in terms of  
Fourier series.
\[ 
x_n = \sum_{k \in \integer^d}  e^{2 \pi i k \cdot \omega n }  
\hat x_k  =  K( n \omega)  
\]
where $K(\theta) = \sum_{k \in \integer^d}  e^{2 \pi i k \cdot \theta } \hat x_k  $. 

We can think of the function $K$ as an embedding of the torus 
$\torus^d$
into phase space. If $\omega$ does not have any resonances 
(i.e. $k \cdot \omega \ne 0$ for $k \in \integer^d \setminus \{0\}$, 
which can always be arranged by reducing $d$ if there is one), 
then $\{ \omega n\}_{n \in \integer}$ is dense on the torus. 
The map $K$ is often called the \emph{hull function}.

If $x_n$ is an orbit of a map, $x_{n+1} = f(x_n)$ we 
see that $K(n\omega + \omega) = f( K(n \omega) )$. 
Since $\{ \omega n \}_{n \in \integer}$ is dense, this is 
equivalent to 
\begin{equation} \label{qp} 
K(\theta + \omega ) = f( K(\theta) ) \ \forall \theta \in \torus^d
\end{equation} 
Hence, we see that the set $K(\torus^d)$, the image of 
the standard torus under the embedding $K$ is invariant under 
$f$. So, it is customary to describe quasi-periodic solutions 
as \emph{invariant tori}.

The problem of given a map finding a quasi-periodic solution 
of frequency $\omega$ can be formulated as finding an embedding $K$ 
solving \eqref{qp}. 
The equation   \eqref{qp}  will be our fundamental tool to characterize 
quasi-periodic orbits.

\subsection{Set-up of the problem. The invariance equation} 

In this section, we describe informally the geometric set up 
and the geometric meaning of the formulation of our problem. 
The precise formulation of the main result of this paper 
(Theorem~\ref{main_theo})   will be presented in 
Section~\ref{sec_main_res}.

We will be mainly concerned with
an analytic family of maps
$f_{\e,\mu}:\M \longrightarrow \M $, such that 
$$f_{\e,  \mu}^*\Omega = \l(\e)\Omega$$ where $\e\in\complex$ is a small
parameter, $\mu\in\Lambda\subseteq\complex^d $ is an internal
parameter (the drift parameter), and $ \l(\e)=1-\e^\alpha$. 

A good example to keep in mind is the dissipative standard map 
presented in \eqref{diss-est-map}. Note that, for $\e = 0$ and for each $\mu$, the maps $f_{0,\mu}$ are symplectic because $\l(0) = 1$. 

The main assumption  in the main Lemma, Lemma~\ref{main_theo}, 
is  that the map $f_{0, \mu_0}$ has an invariant torus
in which the motion is a rotation of frequency $\omega$ 
which is Diophantine (see Definition~\ref{diophantine}). Note that the drift parameter, $\mu$, is chosen to guarantee the persistence of a quasi periodic orbit of a given frequency $\omega$, so we also consider $\mu = \mu_\e$.

Following the discussion in Section~\ref{qp-functions}
and, in particular \eqref{qp}, we see that
finding a quasi-periodic orbit for $f_{\e,\mu_\e}$  is equivalent to 
finding families of embeddings $K_\e$ and families of parameters $\mu_\e$ in such a way that 
\begin{equation} \label{inv} 
f_{\e, \mu_\e}\circ K_\e(\theta) = K_\e(\theta + \omega) 
\end{equation} 

Equation \eqref{inv} should be interpreted as, 
given the family $f_{\e, \mu}$ and the frequency $\omega$ 
finding $\mu_\e, K_\e$. For this work, the sense in which \eqref{inv} is meant to hold is the meaning of formal power series (the coefficients of $\e^n$ on both sides 
of \eqref{inv} are identical for all $n$, as it is customary in the study of Lindstedt series). 

Note that the equation \eqref{inv} is highly underdetermined. 
If $\mu_\e, K_\e$ is a solution,  changing $\theta$ into 
$\theta + \sigma_\e$, we obtain 
that $\mu_\e, \tilde K_\e$ is also a solution where 
$\tilde K_\e(\theta) = K_\e(\theta + \sigma_\e)$. 
This change of variables has the physical meaning of choosing 
a change of origins in the torus.

\subsection{Automatic reducibility} \label{sec_auto_red} 

As it is noted in \cite{Cal-Cel-Lla-13}, a very useful property of conformally symplectic systems is that solutions to equation \eqref{inv} satisfy the so-called \textit{automatic reducibility}, that is, in a neighborhood of an invariant torus, one can find a system of coordinates in which the linearization of the evolution has constant coefficients.

\begin{lemma}\label{lem_red}
Let $f_\mu:\M\longrightarrow\M$, such that, $f_\mu^*\Omega= \lambda \Omega$, and $K:\torus^d \longrightarrow \M$ such that $f_\mu\circ K(\t) = K(\t+\omega)$ with $\omega$ an irrational vector. If $\N= (DK^\top DK)^{-1}$, then, the $2d\times 2d$ matrix \begin{equation}\label{M_first}
    M(\t) = \left[DK(\t)| J^{-1}\circ K(\t) DK(\t)\N(\t)\right]
\end{equation}
satisfies \begin{equation}\label{red_eq_1}
    Df_\mu\circ K(\t)M(\t) = M(\t+\omega)\begin{pmatrix} \Id & S(\t) \\ 0 & \l\Id \end{pmatrix} 
\end{equation}
where $\Id\in \real^{d\times d} $ and  $S(\t)$ is an explicit algebraic expression involving $DK$, $Df_\mu$, $J\circ K$, and, $\N$. 
\end{lemma}

The proof of Lemma \ref{lem_red} is given in \cite{Cal-Cel-Lla-13}. The argument is as follows, taking derivative in equation \eqref{inv} one has $Df_\mu\circ K_0(\t) DK_0(\t) = DK_0(\t+\omega)$ which gives the first column in \eqref{red_eq_1}. The second column comes from the fact that  the conformally symplectic property, $f_\mu^*\Omega = \l\Omega$, implies that the invariant torus given by equation \eqref{inv} is Lagrangian. Then, using the conformally symplectic geometry the second column can be obtained.

\begin{remark}
As it is pointed out in \cite{Cal-Cel-Lla-13} if $K$ is an approximate solution of \eqref{inv}, that is, \begin{equation}\label{inv_error}
     f_\mu \circ K(\t) - K(\t+\omega) =: E(\t)
\end{equation}
the relation \eqref{red_eq_1} will hold with an error, $R$, that can be estimated in terms of the error, $E(\t)$, of the invariance equation, that is \begin{equation}\label{red_error_eq_1}
    Df_\mu\circ K(\t)M(\t) = M(\t+\omega)\begin{pmatrix} \Id & S(\t) \\ 0 & \l\Id \end{pmatrix} +R (\t), 
\end{equation} 
with 
\begin{align}
S (\t) & \equiv  P(\t+\omega)^\top Df\circ K(\t)J^{-1}\circ K (\t)P(\t) - \N (\theta +\omega)^\top\Gamma(\t+\omega) \N  (\t+\omega)\l \label{S_red} \\
P(\t)  & \equiv DK(\t)\N(\t), \nonumber \\
\Gamma(\t) & \equiv DK(\t)^\top J^{-1}\circ K(\t)DK(\t). \nonumber 
\end{align}
Moreover, 
\begin{equation} \label{red_error}
    R(\t)= \left[ DE(\t)\left | V(\t+\omega)(\Tilde{B}(\t) -\l\Id) +DK(\t+ \omega)(\Tilde{S}(\t) -S(\t)) \right]\right.
\end{equation}
where 
\begin{align}
      V(\t) &\equiv J^{-1}\circ K(\t) DK(\t)\N(\t) \label{V} \\
      \Tilde{B}(\t) -\l\Id&\equiv DK(\t)^\top J\circ K(\t)DK(\t)\Tilde{S}(\t) \\
      \Tilde{S}(\t) - S(\t) &\equiv -\N(\t+\omega)^\top \Gamma(\t+\omega)\N(\t+\omega) (\Tilde{B}(\t) - \l\Id) 
\end{align}
We note that $\Tilde{B}-\Id$ is estimated by the norm of \eqref{inv_error}, thus $R$ in \eqref{red_error} can be estimated by the norm of \eqref{inv_error} as it is shown in Lemma \ref{lem_estimate_reducibility}. The derivation of the formulas in \eqref{S_red}, \eqref{red_error}, and \eqref{V} can be found in \cite{Cal-Cel-Lla-13}.
\end{remark}

\begin{remark}
Observe that when considering $K_0$, $\mu_0$ satisfying \eqref{inv} and a perturbation $K_\e$, $\mu_\e$ (which could be given in terms of formal power series), equation \eqref{red_error_eq_1} is also satisfied by $K_\e$, $\mu_\e$ but with all the expressions depending on $\e$ (small enough), that is, $$ Df_{\mu_\e}\circ K_\e(\t)M_\e(\t) = M_\e(\t+\omega)\begin{pmatrix} \Id & S_\e(\t) \\ 0 & \l\Id \end{pmatrix} +R_\e (\t).$$

\end{remark}


\section{Statement of the main result, Theorem \ref{main_theo_dsm} }\label{sec_main_res}

In this section we state the main result, Theorem \ref{main_theo_dsm}, which gives the Gevrey character of the perturbative expansions of the solutions to equation \eqref{inv}. First we introduce  a normalization which guarantees the uniqueness of the solutions to equation \eqref{inv}.

\subsection{Normalization and local uniqueness}

The centerpiece of this work  is the invariance equation \begin{equation}\label{inv_2}
    f_{\e,\mu_\e}\circ K_\e = K_\e \circ T_\omega 
\end{equation}
where $T_\omega(\t)= \t+\omega$. Note that if $(K,\mu)$ is a solution of the invariant equation \eqref{inv_2}, then, for any $\sigma\in \torus^d$, $(K\circ T_\sigma, \mu )$ is also a solution of \eqref{inv_2}, due to the fact that $K\circ T_\sigma$ parameterizes the same torus as $K$. So, in order to get uniqueness it is neccesary to impose a normalization condition.

\begin{definition}
 We say that a torus with embedding $K$ is normalized with respect to $K_0$  when \begin{equation} \label{normal} \int_{\torus^d} \left[ M_0^{-1}(\t)(K(\t) -K_0(\t)) \right]_dd\t =0 
\end{equation}
 where the subscript $d$ indicates that we take the first $d$ rows of the $2d\times d$ matrix, and $M_0$ is constructed from $K_0$ as in \eqref{M_first}.
\end{definition}

We also recall the following result (\cite{Cal-Cel-Lla-13}, Proposition 26) which shows that this condition can be imposed without loss of generality for solutions that are close to one another.

\begin{proposition}\label{prop_giv_norm}
Let $K_0, K$ be solutions of \eqref{inv_2} and $||K -K_0||_{C^1}$ be sufficiently small (with respect to quatities depending only on $M$ -computed out of $K_0$ - and $f$). Then, there exists $\sigma\in\real^d$, such that $K^{(\sigma)}= K\circ T_\sigma $ satisfies \eqref{normal}. Furthermore, $$|\sigma |\leq C||K -K_0||_{C^1}$$ where the constant $C$  can be chosen to be as close to 1 as desired by assuming that $f_\mu$, $K_0$, and $K_1$ are twice differentiable, $DK_0^\top DK_1$ is invertible and $||K -K_0||_{C^0}$ is sufficiently small. The $\sigma$ thus chosen is locally unique.
\end{proposition}

\begin{remark}\label{remark_uniq_ser}
As it is noted in \cite{Cal-Cel-Lla-13} the normalization \eqref{normal} works as well when  $K$ is only an approximate solution. Then, assuming that $K_0$ is a solution of equation \eqref{inv_2}, the normalization condition \eqref{normal} for an approximate solution of \eqref{inv_2} given as power series expansion $\sum_{n=0}^\infty K_n(\t)\e^n$ is equivalent to the conditions \begin{equation}
\int_{\torus^d}\left[ M_0^{-1}(\t)K_n(\t)\right]_dd\t = 0 \label{normal_coeff}
\end{equation} for all  $n\geq 1$.

\end{remark}

\subsection{Main Theorem}

Here we present our main theorem, Theorem \ref{main_theo_dsm}. 

\begin{theorem}[Main Theorem]\label{main_theo_dsm}
Let $\omega\in\D(\nu,\tau) $. Consider the map $f:\torus\times \real \rightarrow \torus\times \real$ given by
 \begin{equation}\label{dsm_theo}
    f_{\e,\mu_\e}(x,y)= (x+\l(\e)y + \mu_\e -\e V'(x) , \l(\e)y + \mu_\e -\e V'(x)) 
\end{equation}
where $\l(\e)=1-\e^\alpha$, $\alpha\in\nat$, $V(x)$ is a trigonometric polynomial, $\mu_\e \in\complex$, and $\e\in\complex$. Then, there exists $\r_0>0$ such that the following holds
\begin{itemize}
\item[(A)] There exist  formal power series expansions $K^{[\infty]}_\e= \sum_{j=0}^\infty K_j\e^j$ and $\mu_\e^{[\infty]}=\sum_{j=0}^\infty \mu_j\e^j $ satisfying $f_{\e,\mu}\circ K =K(\t +\omega)$ in the sense of formal power series. More precisely, defining $\expan{K_\e} N =\sum_{j=0}^N K_j\e^j$ and $\expan{\mu_\e} N=\sum_{j=0}^N \mu_j\e^j$ for any $N\in \nat$ we have \begin{equation}\label{norm_error_0}
\normm{f_{\e,\expan {\mu_\e} N}\circ\expan {K_\e}N -\expan {K_\e}N\circ T_\omega}{\r_0}\leq C_N|\e|^{N+1}.
\end{equation}
where $C_N >0$. Moreover, if the $K_j$'s satisfy the normalization condition \eqref{normal_coeff}, then the expansions $K_\e^{[\infty]}$, $\mu_\e^{[\infty]}$ are unique. 
\item[(B)]
The unique formal power series expansions, $K^{[\infty]}_\e$ and $\mu^{[\infty]}_\e$,  satisfying \eqref{norm_error_0} and the normalization \eqref{normal_coeff}
are such that $K^{[\infty]}\in \G_{\r_0}^{2\tau/\alpha}$ and $\mu^{[\infty]}\in \G^{2\tau/\alpha}$, i.e., there exists constants $L$, $F$, $N_0$ such that \begin{equation} \normm {K_n}{\r_0}\leq LF^nn^{(2\tau/\alpha)n}\quad\mbox{and} \quad \left|{\mu_n}\right|\leq LF^nn^{(2\tau/\alpha)n}\qquad  \mbox{for any }n> N_0. 
\end{equation}
\end{itemize}
\end{theorem}

\smallskip
The proof of  Theorem \ref{main_theo_dsm} is an easy consequence of Lemma \ref{main_theo}. Proposition \ref{hyp-satisfy}, given in the Appendix, shows the hypothesis of Lemma \ref{main_theo} are satisfied for maps of the form \eqref{dsm_theo}. Lemma \ref{main_theo} states the same results as Theorem \ref{main_theo_dsm} but in a more general setting.

\begin{remark}
It is instructive to compare the results in Theorem \ref{main_theo_dsm} with 
the numerical explorations of \cite{Bus-Cal-19} (see also \cite{Bus-Cal-20}). In the case that $\lambda(\e) = 1 - \e^3$ and $\omega$ is the golden mean, Theorem \ref{main_theo_dsm} gives that the expansion satisfies the Gevrey bounds with exponent 2/3. Of course, Theorem \ref{main_theo_dsm} gives only an upper bound and lower exponents could also be true. 
The  numerical results in \cite{Bus-Cal-19} and  \cite{Bus-Cal-20} lead to the conjecture that the expansion $\sum K_n\e^n$ has some well defined asymptotics
\begin{equation}\label{asymptotic-uk} 
\|K_n\|_\rho^{1/n} \approx C n^\sigma
\end{equation} 
with a slightly smaller Gevery exponent, $\sigma \approx 0.3$. 
The asymptotics \eqref{asymptotic-uk} is compatible with the results in Theorem \ref{main_theo_dsm}, but 
suggests that the results in Theorem \ref{main_theo_dsm} are not optimal. We call attention that \cite{Bus-Cal-19} contained an unfortunate typo and the results attributed there to $\|K_n\|^{1/n}$ are actually results for $ \|n! K_n\|^{1/n} $, this is corrected in \cite{Bus-Cal-20}. The paper \cite{Bus-Cal-20} also presents several other patterns in the series (refined versions of \eqref{asymptotic-uk} including oscillations of period $3$, studies for other Diophantine numbers, etc.)  We hope that the method presented in this paper can lead to studies of these phenomena, hitherto discovered only through numerical implementation. 

We think that the argument in Theorem \ref{main_theo_dsm} can optimized to lower the Gevrey exponent and get closer to the numerical values, but, since the method of proof is rather novel, we decided to follow the advice \emph{``Premature optimization is the root of all evil''} \cite{Knuth}, and present the argument in its simplest form so that 
it could, perhaps, be applied to other problems. \\ 
\end{remark}

For the sake of completeness, before stating the main Lemma we will state a Theorem in \cite{Cal-Cel-Lla-16} which assures the existence of formal power series expansions satisfying \eqref{inv_2} up to any order for conformally symplectic systems.

\begin{theorem}[\cite{Cal-Cel-Lla-16}, Theorem 12] \label{theo_prev_paper}
Let $\M\equiv \torus^d \times\B$ with $B\subseteq \real^d$ an open, simply connected domain with smooth boundary; $\M$ is endowed with an analytic symplectic form $\Omega$.

Let $\omega\in \D(\tau,\nu)$ and consider a family $f_{\e,\mu}$ of conformally symplectic mappings that satisfy \begin{equation}
f_{\e,\mu}^*\Omega =\l(\e)\Omega,
\end{equation} with $\mu \in \Lambda, \Lambda\subseteq \complex^d$, $\l(\e)=1-\e^\alpha$, $\alpha\in\nat$ and $\e\in\complex$.

Assume that for $\e=0$ the family of maps $f_{0,\mu}$ is symplectic and that for some value $\mu_0$ the map $f_{0,\mu_0}$ admits a Lagrangian invariant torus, namely we can find an analytic embedding  $K_0\in \A_\r(\torus^d, \M)$, for some $\r>0$, such that \begin{equation} \label{inv_eq}
f_{0, \mu_0}\circ K_0 =K_0\circ T_\omega.
\end{equation}
Furthermore, assume that the torus $K_0$ satisfies the following hypothesis:

\textbf{\emph{HND}} Let the following non-degeneracy condition be satisfied: $$\det\left(\begin{matrix}
\overline{S_0} &\,& \overline{S_0\noaverage{B_{0b}}} +\overline{\tilde{A}_{01}} \\ 0 &\, & \overline{\tilde{A}_{02}} \end{matrix}\right) \neq 0$$ where the $d\times d$ matrix $S_0$ is defined as 
\begin{align*}S_0(\t) &\equiv \N_0(\t+\omega)^T DK_0(\t+\omega) Df_{\mu_0,0} \circ K_0(\t) J^{-1}\circ K_0(\t)DK_0(\t)\N_0(\t) \\
	& -\N_0(\t+\omega)^TDK_0(\t+\omega)^T J^{-1}\circ K_0(\t +\omega)DK_0(\t +\omega)\N_0(\t +\omega)
\end{align*}
with $\N=(DK_0^TDK_0)^{-1}$, the $d\times d$ matrices $\tilde{A}_{01}, \tilde{A}_{02}$ denote the first $d $ and the last $d$ rows of the $2d\times d$ matrix $\tilde{A_0}= \left(M_0 \circ T_\omega\right)^{-1}\left(D_\mu f_{0, \mu_0}\circ K_0\right) $, where $M_0$ is as in \eqref{M_first}, $\noaverage{B_{0b}}$ is the solution (with zero average) of the cohomology equation $\noaverage{B_{0b}} -B_{0b}\circ T_\omega = -\noaverage{\tilde{A}_{02}}$,  where $\noaverage{B_{0b}}\equiv B_{0b} - \overline{B_{0b}}$ and the overline denotes the average.
\\

Then, we have the following
\begin{itemize}
\item[(A)] There exist a formal power series expansions $K^{[\infty]}_\e= \sum_{j=0}^\infty K_j\e^j$ and $\mu_\e^{[\infty]}=\sum_{j=0}^\infty \mu_j\e^j $ satisfying \eqref{inv_eq} in the sense of formal power series. More precisely, defining $\expan{K_\e} N =\sum_{j=0}^N K_j\e^j$ and $\expan{\mu_\e} N=\sum_{j=0}^N \mu_j\e^j$ for any $N\in \nat$ and $\r >0 $, we have \begin{equation}\label{norm_error}
\normm{f_{\e,\expan {\mu_\e} N}\circ\expan {K_\e}N -\expan {K_\e}N\circ T_\omega}{\r'}\leq C_N|\e|^{N+1}.
\end{equation}
for some $0<\r'<\r$ and $C_N >0$. 

Moreover, if we require the  $K_j$'s satisfy the normalization condition \eqref{normal_coeff}, then the expansions $K_\e^{[\infty]}$, $\mu_\e^{[\infty]}$ are unique.
\end{itemize}
\end{theorem}

Note that Theorem~\ref{theo_prev_paper} does not assume that the case $\e = 0$ is an 
integrable system, as it is the case for the map \eqref{dsm_theo}, it suffices that the case $\e = 0$ is a Hamiltonian system with a KAM torus. 

\begin{remark}
Denoting 
    \begin{equation}\label{E^N}
    E_\e^N(\t)\equiv f_{\expan {\e,\mu_\e} N}\circ\expan {K_\e}N(\t) -\expan {K_\e}N(\t+\omega)
    \end{equation} then \eqref{norm_error} can be written as \begin{equation}\label{norm_E^N}
    \normm{E^N_\e}{\r'}\leq C_N|\e|^{N+1}.
    \end{equation} 
According to the notation introduced earlier, this means that $E^N_\e\sim \ord{N+1}$ or $E^N_\e =\sum_{j=N+1}^\infty E_j\e^j$ for $\e$ small enough. We denote $$\expanin{E}N{2N}_\e =\sum_{j=N+1}^{2N}E_j\e^j$$ the trucated series.
\end{remark}
The following lemma, Lemma \ref{main_theo}, can be considered as an improvement of Theorem \ref{theo_prev_paper} in the sense that it gives Gevrey bounds for the coefficients $K_j, \mu_j$ of the unique (under normalization) formal power series expansions $K_\e^{[\infty]}, \mu_\e^{[\infty]}$.

\begin{lemma}[Main Lemma]\label{main_theo}
Assume the hypothesis of Theorem \ref{theo_prev_paper}. Assume also that for any $\e$, small enough, and for any $N\in \nat$ we have:
\begin{itemize}
    \item[\emph{\textbf{HTP1}}]   $\expanin{\tilde{E}_{\e,2}}N{2N}$, $\tilde{A_{\e,2}^N}$ are trigonometric polynomials in $\t$ of degree at most $aN$, $a\in\nat$. Where $\expanin{\tilde{E}_{\e,2}}N{2N}$, $\tilde{A_{\e,2}^N}$ denote the $d\times 1$ and $d\times d$ matrices, respectively, given by taking the last $d$  rows of the $2d\times 1$ matrix  $\expanin{\tilde{E}_\e}N{2N}=\left(\expan {M_\e}N\circ T_\omega\right)^{-1} \expanin{E_\e} N{2N}$ and the $2d\times d$ matrix $ \tilde{A}_\e^N = \left(\expan {M_\e}N\circ T_\omega\right)^{-1} D_\mu f_{\e,\expan \mu N}\circ \expan {K_\e}N $, respectively. $\expan{M_\e}N $ is as in \eqref{M_first} constructed from $\expan{K_\e}N$.

    \item [\emph{\textbf{HTP2}}]
    The $d\times d$  matrix 
    \begin{multline}
    \tilde{E}^N_{\Omega,\e}(\theta)\equiv D\expan {K_\e} N (\t+\omega)^\top J\circ \expan {K_\e} N(\t+\omega) D\expan {K_\e} N (\t+\omega) \\ 
			- D(f_{\e,\expan\mu N}\circ \expan {K_\e} N(\t))^\top J\circ (f_{\e,\expan \mu N}\circ \expan {K_\e} N(\t)) D(f_{\e,\expan\mu N}\circ \expan {K_\e} N(\t))
    \end{multline}
is a trigonometric  polynomial of degree at most $aN$.
\end{itemize}

Then, there exist $\r_0\leq \r'$ such that the unique formal power series expansions, $K^{[\infty]}_\e$ and $\mu^{[\infty]}_\e$,  satisfying \eqref{norm_error} and \eqref{normal_coeff} are such that $K^{[\infty]}\in \G_{\r_0}^{2\tau/\alpha}$ and $\mu^{[\infty]}\in \G^{2\tau/\alpha}$, i.e., there exists constants $L$, $F$, $N_0$ such that \begin{equation} \normm {K_n}{\r_0}\leq LF^nn^{(2\tau/\alpha)n}\quad\mbox{and} \quad \left|{\mu_n}\right|\leq LF^nn^{(2\tau/\alpha)n}\qquad  \mbox{for any }n> N_0. 
\end{equation}
\end{lemma}

\bigskip
The proof of Lemma \ref{main_theo}, given in Section \ref{proof-main-lemma}, is done by means of a Newton like method which acts on finite powers series expansions ($\expan{K_\e}N$, $\expan{\mu_\e}N$), this method is described in the next section. We emphasize that this quasi Newton method takes advantage of the conformally symplectic property (see Definitions \ref{def:conformallysymplectic}) that maps like \eqref{dsm_theo} satisfy. 

We also point out that hypothesis \textbf{HTP1} and \textbf{HTP2} are very natural for the maps considered in Theorem \ref{main_theo_dsm}. The verification of these hypothesis for the dissipative standard map is described in detail in Proposition \ref{hyp-satisfy} of the Appendix. In the general setting in which Lemma \ref{main_theo} is stated, the hypothesis \textbf{HTP1} and \textbf{HTP2} are needed  to be able to get estimates, in balls with center at the origin, for the solutions of the linear equations of the quasi Newton method.

\subsection{Asymptotic estimates for invariance functions} \label{sec-gev-asymp}
The formal power series studied in this paper are asymptotic expansions of 
functions $K_\e, \mu_\e$ constructed in \cite{Cal-Cel-Lla-16}. The functions 
$K_\e, \mu_\e$ are determined by the condition 
that they satisfy the invariance equation \eqref{inv_2} and the normalization \eqref{normal_coeff}. In this section we argue that the same method we use to prove the Gevrey estimates also shows that the formal series defined 
here are asymptotic to the functions $K_\e, \mu_\e$ with very strong estimates in 
the remainder, see Theorem~\ref{gev-asymptotics}. 

We emphasize that the functions $K_\e, \mu_\e$ are not constructed out of the asymptotic expansions by complex analysis methods (Borel summation, resummation of series). They are obtained from the requirement that they satisfy
the invariance equation \eqref{inv_2} and the normalization \eqref{normal_coeff}. 
It is an interesting open question whether some resummation of the asymptotic 
expansions studied here  can produce the functions $K_\e, \mu_\e$. 


The domain of definition of the functions $K_\e, \mu_\e$ is rather subtle. In \cite{Cal-Cel-Lla-16}, it is proved  that the domain of
definition of $K_\e, \mu_\e$  contains a set $\G$ obtained by removing sequence of balls that are dense on curves converging to the origin, in fact, it is rigorously showed that $\G$ is a lower bound on the analyticity domain of the functions $K_\e,\mu_\e$. We also point out that the set $\G$ does not contain any ball centered at the origin. Indeed, the set $\G$ does not contain any sector centered at the origin of width bigger than $\pi/\alpha$, thus the width of the domain is not enough to apply many 
methods of complex analysis related to Phragm\'en-Lindel{\"o}f theory.  In the other direction, 
the paper \cite{Cal-Cel-Lla-16} contains arguments showing that for generic perturbations one should not expect that the domain of analyticity contains  the excluded balls (if the perturbation 
happens to be identically zero one indeed obtains a larger domain).  The paper \cite{Bus-Cal-19} studies numerically the maximal domain of definition of the functions $K_\e, \mu_\e$ for the map \eqref{dsm_theo} using 
a variety of methods including Pade summation and continuation methods. Indeed \cite{Bus-Cal-19} conjectured that the series were Gevrey and this was an important motivation for this paper.

The set $\G$ is determined by asking that $\lambda(\e) $ satisfies a Diophantine condition with respect to $\omega$, more precisely, defining 
\begin{equation} \label{nu-lambda}
\tilde{\nu} = \tilde{\nu}(\l;\omega,\tau) \equiv \sup_{k\in\integer^d\backslash\{0\}}|e^{2\pi i k\cdot\omega} -\l |^{-1}|k|^{-\tau}
\end{equation}
one has
\begin{equation} \label{set-G}
\G = \G(A;\omega,\tau, N) = \left\{ \e \in \complex \; : \quad\tilde{\nu}(\l;\omega,\tau) |\l(\e) -1 |^{N+1} \leq A  \right\}.
\end{equation}

The basic idea to prove the existence of the functions $K_\e, \mu_\e$ is as follows: The formal power expansions produces a sequence of polynomials which satisfy the invariance equation \eqref{inv_2} rather approximately in a ball.  In the intersection of the ball with the set $\G$, we can apply the a-posteriori  theorem, Theorem 14 in \cite{Cal-Cel-Lla-16}, and obtain a true solution of \eqref{inv_2}. Of course, the detailed implementation
requires taking into account several other issues such as the absence of monodromy.



In this paper, we will use a very similar technique. 
As as byproduct of the estimates used in the 
proof of Lemma~\ref{main_theo}, we obtain that some truncations of 
the formal expansion satisfy the invariance equation up to a very small 
error in appropriate balls.  Then, in the intersection of the balls with the set 
$\G$ we will be able to apply Theorem 20 in \cite{Cal-Cel-Lla-13}. 

More precisely we have:  

\begin{theorem}\label{gev-asymptotics}
Assuming the hypothesis of Lemma~\ref{main_theo} and $n\in(2^hN_0,2^{h+1}N_0]\cap\nat$, then for any $0< \d <\r_0$ the asymptotic expansions in Lemma~\ref{main_theo} satisfy 

\begin{equation}\label{gev-asympt-estimate}
\sup_{\e\in\G, |\e|\leq \tilde{\gamma}_{h+2}} \normm{ \sum_{j=1}^{n} K_j\e^j -K_\e}{\r_0-\d}   \leq \left( U +  V  2^{h(3\tau +3d)} r^{n+1}r^{2^hN_0}\right)(CD)^h B^{h^2}r^{(2^h-1)N_0}\normm{E^{N_0}}{\r_0} 
\end{equation}
where $\hat{C}$ and $C$ are uniform constants and $U=\hat{C} \nu^{-1}\tilde{\nu}^{-1}\d^{-2(\tau+d)}$, $V=C \nu^{-3} (aN_0)^{2\tau}\r_0^{-(\tau+3d)} 2^{2\tau+6d}$, $D=\nu^{-6}(aN_0)^{4\tau}\r_0^{-(2\tau +6d)} 2^{-(4\tau +12d)}$, $r= 2^{-\tau/\alpha}$, $ B= 2^{6\tau +6d}$, and $ \tilde{\g}_h = (2^{-1}\nu)^{1/\alpha}(a2^hN_0)^{-\tau/\alpha}.$
\end{theorem}

Note that \eqref{gev-asympt-estimate} can be understood  as having super-exponentially small errors in domains decreasing exponentially fast. It is also important to note that almost all constants in \eqref{gev-asympt-estimate} are given explicitly. The proof of Theorem~\ref{gev-asymptotics} is given in Section \ref{proof-gev-asymptotics}.

\section{Iterative step of the quasi Newton method.} \label{sec_new_step}

The KAM procedure for the proof of Theorem \ref{main_theo} is based on the application of a quasi Newton method, which is described in Section  \ref{newton_method}. Before describing this procedure we introduce two types of cohomology equations that allow us to solve the linear equations, and obtain estimates, of the modified Newton method. The estimates for each step of the method will be given in  Section \ref{sec_nolin_est}.

\subsection{Estimates for some cohomology equations}
The iterative step described in Section \ref{newton_method} depends on the solution of two cohomology equations. The first equation, \eqref{cohom2}, is very standard in KAM theory. The estimate given in Lemma \ref{classic_lin_est} is well known for the experts in KAM theory, we have decided to include a proof here for the sake of completeness. 
The second type of cohomology equation we consider, \eqref{cohom}, it is more complicated to study due to the fact of the appearance of the factor $\l(\e)=1-\e^\alpha$. This factor introduces some restrictions in the set of parameters, $\e$, for which we are able to obtain estimates.  

\subsubsection{Standard cohomology equation}

The first cohology equation we deal with is the following

\begin{equation}\label{cohom2}\varphi_\e(\t)-\varphi_\e(\t+\omega)=\eta_\e(\t) \end{equation}
Lemma \ref{classic_lin_est} below, gives sufficient conditions to solve equation \eqref{cohom2} and to obtain estimates of its solutions. This estimates are very standard in KAM theory.

\lem{classic_lin_est} Let $\omega\in\D(\nu,\tau)$. Assume that $\eta\in\A_{\r,r}$ is such that $\int_{\torus^d}\eta_\e(\t)d\t=0$. Then, we can find a unique solution of \eqref{cohom2}, $\varphi_\e$, that satisfies $\int_{\torus^d}\varphi_\e(\t)d\t=0$. Moreover, if for any $0<\d\leq\r$ we have $\varphi\in\A_{\r-\d,r}$, then $$\norm\varphi{\r-\d}r\leq C\nu^{-1}\d^{-(\tau+d)}\norm \eta \r r.$$ With $C=C(d)$. Furthermore, $\eta_\e\sim\ord k$ implies $\varphi_\e\sim \ord k$.\elem
\begin{proof}
Expanding in Fourier series the solution to \eqref{cohom2} is given by $\varphi_\e(\t) = \sum_{k\in \integer^d\backslash \{0\}}\frac{\eta_k (\e)}{1-e^{2\pi i k\cdot \omega}} e^{2\pi i k\cdot \t} $. Then, using Cauchy estimates one obtains 
\begin{align}
    \normm{\varphi_\e }{\r-\d} 
    &  \leq \sum_{k\in\integer^d\backslash\{0\}} \frac{|\hat{\eta_k}(\e)|}{|1-e^{2\pi ik\cdot\omega}|}\normm{e^{2\pi ik\cdot\t}}{\r-\d} \nonumber \\
    & \leq \sum_{k\in\integer^d\backslash\{0\}} \nu^{-1}|k|^\tau \normm{\eta_\e }{\r}e^{-2\pi |k|\r} e^{2\pi (\r-\d)|k|} \nonumber \\
    & \leq C\nu^{-1}\normm{\eta_\e }{\r} \sum_{j\in \nat} j^{\tau +d + 1}e^{2\pi\d j} \nonumber \\
    & \leq C\nu^{-1}\delta^{-(\tau +d)}\normm{\eta_\e}{\r}.
\end{align}
The last line gives $\varphi_\e \sim \ord{k}$ if $\eta_\e \sim \ord{k}$ and taking supremum over $\e$ the result is proved.
\end{proof}

\begin{remark}
Equation \eqref{cohom2} appears very often in KAM theory. When $\e\in\real$, the paper \cite{Rus-75} contains estimates with a better exponent on $\d$. That is, in the same situation of Lemma \ref{classic_lin_est}, when $\e\in\real$, one can get $\|\varphi_\e\|_{\r-\d} \leq C\nu \d^{-\tau}\|\eta_\e\|_\r $.
\end{remark}

\subsubsection{Parametric cohomology equation}

The second cohomology equation  we are interested in is an equation for $\varphi_\e:\torus^d \rightarrow\complex$, of the form  \begin{equation}\label{cohom} \l(\e)\varphi_\e(\t)-\varphi_\e(\t+\omega)=\eta_\e(\t)\end{equation} where  $\eta_\e :\torus^d \rightarrow \complex$ and $\omega\in\real^d$ are given, $\e$ fixed. 

Note that, as it is seen in Lemma \ref{lin-est}, solve equation \eqref{cohom} presents a small divisors problem. In this case the small divisors depend on the variable $\e$, that is, equation $\eqref{cohom}$ is not expected to have a solution when $ \l(\e) = e^{2\pi i k \cdot \omega} $.  One approach that has been used to deal with the small divisors in equation \eqref{cohom} (see \cite{Cal-Cel-Lla-13}) requires to \emph{remove} a set from the complex plane, $\e\in \complex$, where the denominators $\l(\e) - e^{2\pi i k\cdot \omega}$ are small. This gives rise to a set with a complicated structure, $\G\subset \complex$, of parameters, $\e$,  in which is possible to find a solution, and estimates, of equation \eqref{cohom}. One of the properties of the set $\G$ described in  \cite{Cal-Cel-Lla-13}, is that it does not contain any ball with center at the origin. This property is one of the reasons for which we follow a different approach to deal with equation \eqref{cohom}, to prove the Gevrey estimates in Lemma \ref{main_theo} we rely heavily on being able to obtain estimates of \eqref{cohom} for $\e$ in a ball centered at the origin.

The following two Lemmas allow us to obtain estimates in balls centered at $\e=0$ for the solution, $\varphi_\e$, of equation \eqref{cohom} whenever $\eta_\e $ is a trigonometric polynomial. If the degree of the trig polynomial, $\eta_\e$, is $aN$, Lemma \ref{no-sdiv} gives a relation between this degree and a domain in which the solution, $\varphi_\e$, of \eqref{cohom}  will be analytic in $\e$. 

Note that the requirement of hypothesis \textbf{HTP1} and \textbf{HTP2} in Lemma \ref{main_theo} is due to the fact that the quantities given in these hypothesis will be the right hand side of equations of the form \eqref{cohom}.                                    

\lem{no-sdiv} Let $\omega\in\D(\nu,\tau)$, $\l(\e)=1-\e^\alpha$, $\alpha\geq 1$, and $a,N\in\nat$. If  $|\e| \leq \left(\frac{\nu}{2}\right)^{1/\alpha} \frac{1}{(aN)^{\tau/\alpha}}$, then,  for $|k|\leq aN$ we have  $$\left|\l(\e)-e^{2\pi ik\cdot\omega}\right|\geq\frac{\nu}{2}\frac{1}{(aN)^{\tau}}$$
\elem
\begin{proof}
$$|e^{2\pi ik\cdot\omega}-\l(\e)|\geq |e^{2\pi i k\cdot\omega}-1|-|1 -\l(\e)|\geq \frac{\nu}{|k|^\tau}-|\e|^\alpha\geq \frac{\nu}{(aN)^\tau}-\frac{\nu}{2(aN)^\tau}=\frac{\nu}{2}\frac{1}{(aN)^\tau}$$
\end{proof}

\lem {lin-est} Let $\l(\e)=1-\e^\alpha$, $\alpha\geq 1$, $\omega\in\D(\nu,\tau)$; $a,N\in\nat$,  and define $$\g_N=\left(\frac{\nu}{2}\right)^{1/\alpha}\frac{1}{(aN)^{\tau/\alpha}}.$$ Let $\eta\in\A_{\r,\g_N}$ such that $\int_{\torus^d}\eta_\e(\t)d\t=0$ and assume that, for any $\e$,  $\eta_\e(\t)$ is a trigonometric polynomial of degree $aN$ in $\theta$. Then, for any $|\e|\leq\g_N$  equation \eqref{cohom} has a unique solution, $\varphi_\e(\t)$, such that $\int_{\torus^d}\varphi_\e(\t)d\t=0$. Furthermore, if for any $0<\delta\leq\rho$ we have $\varphi\in\A_{\r-\d,\g_N}$, then, $$\norm\varphi{\r-\d}{\g_N}\leq C\nu^{-1}(aN)^\tau\d^{-d}\norm\eta\r{\g_N}.$$ Moreover, if $\eta_\e\sim\ord k$, then $\varphi_\e\sim \ord k$.
\elem

\begin{proof}
Expanding $\eta_\e$ in Fourier series as $\eta_\e(\t)=\sum_{0<|k|\leq aN}\hat{\eta}_k(\e)e^{2\pi ik\cdot\t}$ a solution to \eqref{cohom} is given by $$\varphi_\e(\t)=\sum_{0<|k|\leq aN}\frac{\hat{\eta}_k(\e)}{\l(\e)-e^{2\pi ik\cdot\omega}}e^{2\pi ik\cdot\t}.$$ Using Lemma \ref{no-sdiv} and Cauchy estimates, one obtains that for any $|\e|\leq\g_N$
\begin{align}
\normm{\varphi_\e}{\r-\d} & \leq \sum_{0<|k|\leq aN}\frac{|\hat{\eta_k}(\e)|}{|\l(\e)-e^{2\pi ik\cdot\omega}|}\normm{e^{2\pi ik\cdot\t}}{\r-\d} \nonumber \\ 
		&\leq 2(aN)^\tau\nu^{-1}	\sum_{0<|k|\leq aN}|\hat{\eta}_k(\e)|e^{2\pi|k|(\r-\d)} \nonumber\\
        & \leq 2(aN)^\tau\nu^{-1}\sum_{0<|k|\leq aN}\normm{\eta_\e}\r e^{-2\pi|k|\r}e^{2\pi|k|(\r-\d)} \nonumber\\
        & \leq 2(aN)^\tau\nu^{-1}\normm{\eta_\e}\r\sum_{j=1}^{aN}j^{d-1}e^{-2\pi j\d}\nonumber\\
        & \leq C\nu^{-1}(aN)^\tau \delta^{-d}\normm {\eta_\e}\r \label{phi_eta}
\end{align}

Thus, $\norm \varphi{\r-\d}{\g_N}\leq C\nu^{-1}(aN)^\tau\d^{-d}\norm\eta\r{\g_N}$. The last claim comes from \eqref{phi_eta}, that is $\varphi_\e\sim\ord k$ if $\eta_\e\sim\ord k$.
\end{proof}

\subsection{Formulation  of the quasi Newton method}\label{newton_method}

Every step of the quasi Newton method  starts with a solution of equation \eqref{inv_2} up to order $\e^N$. That is, assume that $$ \expan {K_\e} N(\t)=\sum_{n=0}^NK_n(\t)\e^n,\quad \expan {\mu_\e} N=\sum_{n=0}^N \mu_n\e^n$$ satisfy the normalization \eqref{normal_coeff} and $$  f_{\expan {\e,\mu_\e} N}\circ\expan {K_\e}N(\t) -\expan {K_\e}N(\t+\omega)=:E_\e^N(\t) $$ with $$\normm{E^N_\e}{\rho} \leq C|\e|^{N+1}. $$

\begin{remark}
The first step of the Newton method could start with $\expan K {N_0}$, $\expan \mu {N_0}$, given by Theorem \ref{theo_prev_paper}, for some $N_0$.
\end{remark}

Newton's method consists in finding corrections $\Delta_\e,\mu_\e$ to $\expan{K_\e}N$ and $\expan{\mu_\e}N$ such that the linear approximation of equation \eqref{inv_2} associated to $\expan{K_\e}N+\Delta_\e, \expan{\mu_\e}N +\sigma_\e $  reduces the error up to quadratic terms. Taking into account that $$f_{\e,\mu+ \sigma}\circ (K +\Delta) = f_{\e,\mu}\circ K +\left[Df_{\e,\mu}\circ K\right]\Delta +\left[D_\mu f_{\e,\mu}\circ K \right]\sigma +O(\|\Delta\|^2) +O(\|\sigma\|^2) $$ the Newton equation is  
\begin{equation}\label{newton}
    \left[Df_{\expan {\e,\mu_\e} N}\circ \expan{K_\e}N\right]\Delta_\e -\Delta_\e\circ T_\omega + \left[D_\mu f_{\e, \expan {\mu_\e}N}\circ \expan{ K_\e}N \right]\sigma_\e =-E_\e^N. 
\end{equation}
Equation \eqref{newton} is not easy to solve due to the fact that $Df_{\e,\expan{\mu_\e}{N}}\circ\expan{K_\e }{N}$ is not constant. Following an approach similar to that in \cite{Cal-Cel-Lla-13}, we will not solve \eqref{newton} exactly but we will find approximate solutions that will reduce quadratically the error. The idea is to approximate the solution of \eqref{newton} using the geometric identities introduced in Section \ref{sec_auto_red}. Considering the change of variables \begin{equation}\label{W_def}
    \Delta_\e = \expan{M_\e}N W_\e,
\end{equation} where $\expan{M_\e }N$ is as in \eqref{M_first} computed from $\expan{K_\e}N$. Using \eqref{red_error_eq_1} one obtains that \eqref{newton} is equivalent to \begin{equation}\label{newton_2}
    \expan {M_\e} N\circ T_\omega\left[\left(\begin{matrix} \Id & \expan {S_\e}N \\ 0 & \l(\e)\Id\end{matrix}\right)W_\e -W_\e\circ T_\omega\right] +\left(D_\mu f_{\e,\expan {\mu_\e} N}\circ \expan {K_\e}N\right)\sigma_\e= -E_\e^N -\expan {R_\e} N W_\e
\end{equation}
where $\expan {R_\e} N$ is the error \eqref{red_error} and $\expan {S_\e}N $ is given in 
\eqref{S_red}, both computed from $\expan{K_\e}N$. That is

\beq{M} \expan {M_\e} N\equiv\left[\expan{DK_\e}N\bmid J^{-1}\circ \expan {K_\e} ND\expan {K_\e} N\expan{\N_\e} N\right] \sim \O(|\e |^0)\eeq
\begin{align}
    \expan {S_\e} N & \equiv  {\expan {P_\e}N}^\top Df_{\expan {\mu_\e} N,\e}\circ \expan {K_\e} N J^{-1}\circ\expan {K_\e} N \expan {P_\e}N -\l(\e){\expan {\N_\e} N}^\top\expan {\Gamma_\e} N \expan {\N_\e} N \sim \O(|\e |^0)\label{S} 
    \\  
    \expan {\N_\e}{N} & \equiv \left[\left(D\expan {K_\e} N \right)^\top D\expan {K_\e} N\right]^{-1} \sim \O(|\e |^0),\label{N} 
    \\
    \expan {P_\e}N  & \equiv D\expan {K_\e} N \expan {\N_\e} N, \nonumber 
    \\
    \expan {\Gamma_\e} N & \equiv {D\expan {K_\e} N}^T J^{-1}\circ \expan {K_\e} N D\expan {K_\e} N \label{Gamma_e}
\end{align}
Since we expect both $W_\e$ and $\expan{R_\e}N$ to be estimated by $E^N_\e$, see \eqref{R} and \eqref{w_final_estimate}, the term $W_\e\expan{R_\e}{N}$ is quadratic in $E^N_\e$, thus, we expect that omitting this term in \eqref{newton_2} will not change the quadratic nature of the method. 

In order to be able to get estimates of solutions of  cohomology equations of the form \eqref{cohom} instead of considering the  whole error  $E^N_\e=\sum_{j=N+1}^\infty E_j\e^j$ we only consider a truncation of this series, that is, we only consider $\expanin {E_\e}N{2N}=\sum_{j=N+1}^{2N}E_j\e^j$.\\ 
Taking the above into account our quasi Newton step consist in solving the following equation \beq{quasi_newton}\expan {M_\e}N\circ T_\omega\left[\left(\begin{matrix} \Id & \expan {S_\e}N \\ 0 & \l(\e)\Id\end{matrix}\right)W_\e-W_\e\circ T_\omega\right] +\left(D_\mu f_{\e,\expan {\mu_\e} N}\circ \expan {K_\e}N\right)\sigma_\e= -\expanin {E_\e}N{2N} \eeq
\begin{remark}
The election  of the truncation $\expanin{E_\e}N{2N}$ in \eqref{quasi_newton}  has two very important implications for the proof of our result. The first one is that this will yield a new approximate solution which reduces the error quadratically, as a function of $\e$. Moreover, our model example, the dissipative standard map \eqref{diss-est-map}, will satisfy hypothesis \emph{\textbf{HTP1}} and \emph{\textbf{HTP2}} in Lemma \ref{main_theo} due to the fact that the truncation is made. See appendix \ref{appendix}.
\end{remark}

In order to construct a solution of equation \eqref{quasi_newton}, we follow a similar approach as in \cite{Cal-Cel-Lla-13}. Defining
\beq{etilde}
    \expanin{\tilde{E_\e}}N{2N}:=\left(\expan {M_\e}N\circ T_\omega\right)^{-1}\expanin {E_\e}N{2N}\sim \O(|\e |^{N+1})
\eeq 
\beq{atilde}  
    \tilde{A}^N_\e := \left(\expan {M_\e}N\circ T_\omega\right)^{-1} D_\mu f_{\e,\expan {\mu_\e} N}\circ \expan {K_\e}N \sim \O(|\e |^0)
\eeq
and writing  $\expanin{\tilde{E}_\e}N{2N}\equiv(\expanin {\tilde{E}_{\e,1}}N{2N}, \expanin {\tilde{E}_{\e,2}}N{2N})^\top$, where $\expanin {\tilde{E}_{\e,1}}N{2N}$ and $\expanin {\tilde{E}_{\e,2}}N{2N}$ are the first and last $d$ rows of the $2d\times 1$ matrix $\expanin{\tilde{E_\e}}N{2N}$. Similarly, write $\tilde{A}_\e^N=(\tilde{A}^N_{\e,1}, \tilde{A}^N_{\e,2})^\top$ and $W_\e=(W_{\e,1},W_{\e,2})^\top$. Then \eqref{quasi_newton} can be written in components as 
\begin{align}
    W_{\e,1} -W_{\e,1}\circ T_\omega & = -\expan{S_\e }{N}W_{\e,2} - \expanin{\tilde{E}_{\e,1}}N{2N} - \tilde{A}^N_{\e,1}\sigma_\e  \\
    \lambda(\e)W_{\e,2} -W_{\e,2}\circ T_\omega & =  - \expanin{\tilde{E}_{\e,2}}N{2N} - \tilde{A}^N_{\e,2}\sigma_\e
\end{align}
Denoting $\overline{W_{\e,i}}$ as the average of $W_{\e,i}$, with respect to $\t$, and  $ \noaverage {W_{\e,i}}= W_{\e,i} - \overline{W_{\e,i}}$, $i=1,2$; we can divide the system above into two systems, one for the average and another one for the no-average part, that is
\begin{align}
    0 &= -\overline{\expan{S_\e }{N}}\overline{W_{\e,2}} -\overline{\expan{S_\e}{N}\noaverage{W_{\e,2}}} -\overline{\expanin{\tilde{E}_{\e,1} }N{2N}} -\overline{\tilde{A}_{\e,1}^N} \sigma_\e \nonumber \\
    \e^3 \overline{W_{\e,2}} &= -\overline{\expanin{\tilde{E}_{\e,2} }N{2N}} -\overline{\tilde{A}_{\e,2}^N} \sigma_\e \label{syst_1}
\end{align}

\begin{align}
      \noaverage{W_{\e,1}} -\noaverage{W_{\e,1}}\circ T_\omega & = -\noaverage{\expan{S_\e }{N}W_{\e,2}} - \noaverage{\expanin{\tilde{E}_{\e,1}}N{2N}} - \noaverage{\tilde{A}^N_{\e,1}}\sigma_\e \nonumber \\
    \lambda(\e)\noaverage{W_{\e,2}} -\noaverage{W_{\e,2}}\circ T_\omega & =  - \noaverage{\expanin{\tilde{E}_{\e,2}}N{2N}} - \noaverage{\tilde{A}^N_{\e,2}}\sigma_\e.\label{syst_2}  
\end{align}
In order to uncouple systems \eqref{syst_1} and \eqref{syst_2} we consider $\noaverage{W_{\e,2}}$ as an affine function of $\sigma_\e$, due to \eqref{syst_2}. That is,   

\beq{nonaverage_w2} \noaverage{W_{\e,2}}=\noaverage{B_{a,\e}}+\noaverage{B_{b,\e}} \sigma_\e \eeq
where $\noaverage{B_{a,\e}}$ and $\noaverage{B_{b,\e}}$ are defined as the solutions of 
  \beqa{ba} \l(\e)\noaverage{B_{a,\e}}-\noaverage{B_{a,\e}}\circ T_\omega= -\noaverage{\expanin{\tilde{E}_{\e,2}}N{2N} } \\ \label{bb}
\l(\e)\noaverage{B_{b,\e}}-\noaverage{B_{b,\e}}\circ T_\omega= -\noaverage{\tilde{A}^N_{\e,2}}. \eeqa
Due to \textbf{HTP1}, and applying Lemma \ref{lin-est}, equations \eqref{ba} and \eqref{bb} can be solved and we can get estimates in balls with center at $\e=0$. Once that \eqref{ba} and \eqref{bb} are solved, and using \eqref{nonaverage_w2}, system \eqref{syst_1} can be written as 
\beq{w2barsigma} 
\left(\begin{matrix} 
\overline{\expan{S_\e}N} &\,& \overline{\expan {S_\e}N\noaverage{B_{b,\e}}}+\overline{\tilde{A}_{\e,1}^N} \\ \e^3\Id &\,& \overline{\tilde{A}_{\e,2}^N}
\end{matrix}\right) 
\left(\begin{matrix}\overline{W_{\e,2}} \\ \sigma_\e \end{matrix}\right) = 
\left(\begin{matrix} -\overline{\expan {S_\e}N\noaverage{B_{a,\e}}}-\overline{\expanin{\tilde{E}_{\e,1}}N{2N}} \\ -\overline{\expanin{\tilde{E}_{\e,2}}N{2N}} \end{matrix}\right) \eeq 

\begin{remark} 
Due to \emph{\textbf{HND}} in Theorem \ref{theo_prev_paper} the matrix in the left hand side of \eqref{w2barsigma} is invertible at $\e=0$. By the continuity of the determinant, equation \eqref{w2barsigma} can be solved for $\e$ small enough and the inverse is analytic in $\e$. 
\end{remark}
Thus, \eqref{nonaverage_w2} and \eqref{w2barsigma} yield $\sigma_\e\sim \ord{N+1}$ and $W_{\e,2} = \noaverage {W_{\e,2}} + \overline{W_{\e,2}} \sim \ord{N+1}$. It remains to find $W_{\e,1}$, this can be done by solving the equation \beq{w1} \noaverage{W_{\e,1}}-\noaverage{W_{\e,1}}\circ T_\omega= -\noaverage{\expan {S_\e}NW_{\e,2}} -\noaverage{\expanin{\tilde{E}_{\e,1}}N{2N}} -\noaverage{A^N_{\e,1}}\sigma_\e, \eeq 
which can be done due to Lemma \ref{classic_lin_est}. To fulfill the normalization condition \eqref{normal_coeff} and obtain uniqueness of the coefficients of the perturbative expansions, $\overline{W}_{\e,1}$ is chosen as  
\begin{equation}\label{uni-cond}
\overline{W}_{\e,1} = -\left( \int_{\torus^d} \left[ M_0^{-1}(\t)D\expan KN_\e \right]_d d\t\right)^{-1}  \int_{\torus^d} \left[M_0^{-1}(\t)\left( D\expan KN _\e\noaverage {W_{\e,1}} + \expan {V_\e}N W_{\e,2}\right)\right]_d d\t
\end{equation} 
where $\expan{V}{N} = J^{-1}\circ \expan {K_\e} ND\expan {K_\e} N\expan{\N_\e} N$ is the second \textit{column} of the matrix $\expan{M_\e}{N}$, see Remark \ref{remark_uniq_ser}.

\begin{remark}

Assuming that $\expan{K_\e }{N}$ satisfies the normalization \eqref{normal_coeff}, then the new approximation $\expan{K_\e}{N}+ \Delta_\e$ will satisfy \eqref{normal_coeff} if the correction 
satisfies $$\int_{\torus^d} M_0^{-1}(\t)\Delta_\e(\t) d\t =0.$$ Since $\Delta_\e =\expan {M_\e}N W_\e = D\expan {K_\e}N W_{\e,1} + \expan {V_\e}N W_{\e,2}  = D\expan {K_\e}N  \left( \noaverage{W_{\e,1}} +\overline{W_{\e,1}}\right) + \expan {V_\e}N W_{\e,2}$, \eqref{uni-cond} follows from the fact that $\int_{\torus^d} \left[ M_0^{-1}D\expan {K_\e}N \overline{W}_{\e,1}\right]_d d\t = \int_{\torus^d} \left[M_0^{-1}D\expan {K_\e}N\right]_d d\t \overline{W_{\e,1}}$.   
Note that the $d\times d$ matrix $\int_{\torus^d} \left[M_0^{-1}(\t)D\expan KN_\e(\t) \right]_d d\t$ 
is invertible, for $\e$ small enough,  due to the fact that $D\expan KN_\e(\t)$ is a perturbation of $DK_0(\t)$ and $\left[M^{-1}_0(\t)DK_0(\t)\right]_d =I_{d\times d}$, because $M_0(\t)  = \left[ DK_0(\t)|V_0(\t) \right]$.
\end{remark}

This yields, $W_{\e,1} =\noaverage{W_{\e,1}} +\overline{W}_{\e,1}\sim\ord{N+1}$ and thus  \beq{Delta}  \Delta_\e= \expan{M_\e} N W_\e \sim\ord {N+1}\quad\mbox{and}\quad \sigma_\e\sim\ord {N+1}.\eeq which means that $\Delta_\e =\sum_{n=N+1}^\infty \Delta_n\e^n$ and $\sigma_\e=\sum_{n=N+1}^\infty \sigma_n\e^n$. Finally, we take the corrections as 
\begin{equation}
\expanin{\Delta_\e} N{2N} \equiv \sum_{n=N+1}^{2N}\Delta_n\e^n\quad\mbox{and}\quad \expanin\sigma N{2N}_\e\equiv\sum_{n=N+1}^{2N}\sigma_n\e^n.
\end{equation}  
Therefore, the new approximation is chosen as 
\begin{equation}
\expan {K_\e}{2N}:= \expan K{N}_\e +\expanin{\Delta_\e}N{2N}\quad \mbox{and} \quad \expan{\mu_\e}{2N}:= \expan {\mu_\e} N +\expanin {\sigma_\e}N{2N}.
\end{equation}

\begin{remark}
Notice that, due to Lemma \ref{lin-est}, the solutions of \eqref{ba} and  \eqref{bb} will satisfy  $\noaverage {B_{a,\e}} \sim \O(|\e|^{N+1})$ and $\noaverage{B_{b,\e}} \sim \O(|\e|^0)$, because $\noaverage{\expanin{\tilde{E}_{\e,2}}N{2N} }\sim\O(|\e|^{N+1})$ and  $\noaverage{\tilde{A}^N_{\e,2}}\sim \O(|\e|^0)$. Moreover, \eqref{w2barsigma} implies that $\overline{W}_{\e,2}\sim\O(|\e|^{N+1})$ and $\sigma_\e\sim \O(|\e|^{N+1})$. Thus, $W_{\e,2}\sim\O(|\e|^{N+1})$ and similarly $W_{\e,1}\sim \ord {N+1}$ which implies $\Delta_\e\sim\ord {N+1}$.
\end{remark}

\subsection{Algorithm for the iterative step}
The procedure described above leads Algorithm \ref{algo} for a given Diophantine vector $\omega$ and assuming that we are given an analytic family $f_{\e,\mu_\e}$. Some steps in the algorithm are denoted as $p\leftarrow q$, meaning that the quantity $q$ is assigned to the variable $p$.

\begin{algorithm}\label{algo} Given  $\expan{K_\e}N :\torus^n\rightarrow \M$, $\expan{\mu_\e}{N}\in \real^d$. We perform the following computations:

\bigskip
$\begin{array}{ll}
    (1) & E^N_\e \leftarrow f_{\e,\expan{\mu_\e}N} \circ \expan{K_\e }{N} -\expan{K_\e }{N}\circ T_\omega  \\
    (2) & \expanin{E_\e }N{2N}\mbox{ obtained from } E^N_\e \mbox{ by truncation} \\
    (3) & \alpha_\e \leftarrow D\expan{K_\e }{N} \\
    (4) & \N_\e \leftarrow [\alpha_\e^\top \alpha_\e]^{-1} \\
    (5) & V_\e \leftarrow J^{-1}\circ \expan{K_\e}{N}\alpha_\e \N_\e \\
    (6) & M_\e \leftarrow [\alpha_\e | V_\e] \\
    (7) & \beta_\e \leftarrow (M_\e\circ T_\omega)^{-1} \\
    (8) & \expanin{\tilde{E}_\e}N{2N} \leftarrow \beta_\e \expanin{E_\e}N{2N} \\
    (9) & P_\e \leftarrow \alpha_\e \N_\e \\
        & \Gamma_\e \leftarrow \alpha_\e^\top J^{-1}\circ\expan{K_\e }{N} \alpha_\e \\
        & S_\e \leftarrow (P_\e\circ T_\omega)^\top Df_{\expan{\mu_\e }{N}, \e}\circ \expan{K_\e }{N}J^{-1}\circ \expan{K_\e }{N} P_\e -\l(\e)(\N_\e \circ T_\omega)^\top  \Gamma_\e\circ T_\omega (\N_\e \circ T_\omega) \\
        & \tilde{A}_\e \leftarrow \beta_\e D_\mu f_{\expan{\mu_\e }{N}}\circ \expan{K_\e }{N} \\
    (10)& \noaverage{B _{a,\e}} \mbox{ solves } \l(\e)\noaverage{B _{a,\e}} -\noaverage{B _{a,\e}}\circ T_\omega = -\noaverage{\expanin{\tilde{E}_{\e,2}}N{2N}} \\
        & \noaverage{B_{b,\e}}\mbox{ solves } \l(\e)\noaverage{B _{b,\e}} -\noaverage{B _{b,\e}}\circ T_\omega = - \noaverage{\tilde{A}_{\e,2}} \\
    (11)& \mbox{Find } \overline{W_{\e,2}}, \sigma_\e \mbox{ by solving } \\
     & \left(\begin{matrix} 
\overline{S_\e} &\, & \overline{ S_\e\noaverage{B_{b,\e}}}+\overline{\tilde{A}_{\e,1}} \\ \e^3\Id & \, & \overline{\tilde{A}_{\e,2}}
\end{matrix}\right) 
\left(\begin{matrix}\overline{W_{\e,2}} \\ \sigma_\e \end{matrix}\right) = 
\left(\begin{matrix} -\overline{S_\e \noaverage{B_{a,\e}}}-\overline{\expanin{\tilde{E}_{\e,1}}N{2N}} \\ -\overline{\expanin{\tilde{E}_{\e,2}}N{2N}} \end{matrix}\right) \\
    (12)& \noaverage{W_{\e,2}} = \noaverage{B_{a,\e}} +  \noaverage{B_{b,\e}}\sigma_\e \\
    (13)& W_{\e,2} = \noaverage{W_{\e,2}} + \overline{W_{\e,2}} \sim \ord{N+1} \\
    (14)& \noaverage{W_{\e,1}}\mbox{ solves } \noaverage{W_{\e,1}} -\noaverage{W_{\e,1}}\circ T_\omega = -\noaverage{S_\e W_{\e,2}} - \noaverage{\expanin{\tilde{E}_{\e,1}}N{2N}} - \noaverage{\tilde{A}_{\e,1}} \\
    (15)& \overline{W_{\e,1}} =  -\left( \int_{\torus^d} \left[ M_0^{-1}\alpha_\e \right]_1 d\t\right)^{-1}  \int_{\torus^d} \left[M_0^{-1}\left(\alpha_\e\noaverage {W_{\e,1}} + V_\e W_{\e,2}\right)\right]_1 d\t \\
    (16)& W_{\e,1}= \noaverage{W_{\e,1}} + \overline{W_{\e,1}} \sim \ord{N+1} \\
    (17)& \Delta_\e \leftarrow M_\e W_\e \\
    (18)& \expan{K_\e }{2N} \leftarrow \expan{K_\e }{N} + \expanin{\Delta_\e}N{2N} \\
        & \expan{\mu_\e }{2N} \leftarrow \expan{\mu_\e }{N} + \expanin{\sigma_\e }N{2N}
\end{array}$
\end{algorithm}

It is worth to know that all the operations in Algorithm \ref{algo} could be implemented in a few lines in a high level computer language.

\begin{remark}
Note that Algorithm \ref{algo} involves only algebraic operations, compositions, derivatives, truncations, and  solving cohomology equations. This implies that if we start with analytic functions then the output will be an analytic function. 
\end{remark}

\begin{remark}\label{rem_coef}
Note that at each step of the iterative procedure obtained by the quasi Newton method the input will be polynomials of degree $N$ in $\e$, $\expan {K_\e}N\equiv \sum_{n=0}^N K_n\e^n$, and $\expan \mu N_\e= \sum_{n=0}^N \mu_n\e^n$. The output will be  polynomials of degree $2N$ in $\e$ given by 
$$\expan {K_\e}{2N}:=\expan {K_\e}N +\expanin {\Delta_\e} N{2N}\quad\mbox{and}\quad \expan{\mu_\e} {2N} :=\expan{\mu_\e} N +\expanin{\sigma_\e} N{2N}.$$  
Since, by construction, $\expanin {\Delta_\e} N{2N}\sim \ord{N+1}$ and $\expanin {\sigma_\e} N{2N} \sim \ord{N+1}$, the first $N$ coefficients $K_1,K_2,...,K_N$ of the expansion of  $\expan K {2N}$ will be the same coefficients of  $\expan K {N}$ and they will not change for any of the next steps. The same also happens for the coefficients of $\expan{\mu_\e }{2N}$. This is a crucial step for proving the main lemma, Lemma \ref{main_theo}, since due to the fact that the coefficient up to order $N$ do not change after $\log_2(N)$ steps of the modified Newton method, one can use Cauchy estimates in the domains given by Lemma \ref{lin-est} after $\log_2(N)$ steps to obtain estimates on the $N$ coefficient.
\end{remark}

\begin{remark}
To iterate the modified Newton method in Algorithm \ref{algo} it is needed that the new error $E_\e^{2N}$ obtained using the new approximations $\expan {K_\e} {2N} = \expan{K_\e }{N} + \expanin{\Delta_\e }N{2N}$ and  $\expan {\mu_\e} {2N} =  \expan{\mu_\e }{N} + \expanin{\sigma_\e }{N }{2N}$ satisfies $E_\e^{2N} \sim \ord{2N+1}$. This is a consequence of the fact that the new error is quadratic in the original error, as an expansion on $\e$, and this is verified in Lemma \ref{lem_taylor_est}. 


\end{remark}
\section{Estimates for the iterative step.} \label{sec_nolin_est}

In this section we present the estimates for the corrections given by the  Newton step described in Section \ref{sec_new_step}, these estimates are obtained by following the steps in Algorithm \ref{algo}. Throughout this section we consider maps in the spaces $\Aa\r\g$. In the following we will be dealing with equations of the form \eqref{cohom} which, accordingly with Lemma \ref{lin-est}, can be solved if

\begin{equation}\label{gamma}
\e \leq \g_N :=\left(\frac{\nu}{2}\right)^{1/\alpha}\frac{1}{(aN)^{\tau/\alpha}}.
\end{equation}
where $aN$ is the degree of the trigonometric polynomial in the right hand side of \eqref{cohom}.

\subsection{Estimate for the  reducibility error.}

The following Lemma provides an estimate for the error in the approximate reducibility given by $\expan {R_\e} N$ as in \eqref{red_error} computed from $\expan{K_\e}N$. The estimates are obtained by studying qualitatively the geometric identities introduced in Section \ref{sec_auto_red} and taking into account the uniformity on the variable $\e$. 

\begin{lemma}\label{lem_estimate_reducibility}
Let $N\in \nat$, $\omega\in\D(\nu,\tau)$ and $f_{\e,\mu}:\M\rightarrow\M$ be a family of analytic conformally symplectic maps, with $ f_{\e,\mu}^*\Omega = \l(\e)\Omega$, $\mu\in \Lambda\subseteq\complex^d$. Let $\expan {K}{N}\in\Aa\r{\g_N}$ such that $\expan{K_\e} N:\torus^d\rightarrow\M$ is an embedding for any $|\e| \leq \g_N$. Assume also that, for any $|\e| \leq \g_N$,
\begin{itemize}
    \item[\textit{i)}] $\expan {K_\e} N\left(\torus_\r^d\right)\subset \domain (f_{\expan {\e,\mu_\e} N})$ and that there exist $\xi\geq 0$ such that $$\dist\left(\expan {K_\e} N \left(\torus_\r^d\right),\partial\domain (f_{\expan {\e,\mu_\e} N})\right)\geq \xi>0$$
    $$\dist\left(\expan{\mu_\e}{N}, \partial \Lambda\right)\geq \xi> 0$$
    \item[\textit{ii)}] The approximate invariance equation holds $$f_{\e,\expan {\mu_\e} N}\circ \expan {K_\e} N -\expan {K_\e} N\circ T_\omega =E_\e^N \sim \ord{N+1}$$
    \item[\textit{iii)}] \begin{equation}\label{small_cond} \nu^{-1}(aN)^\tau\d^{-(d+1)}\norm{E^N} \r{\g_N} \ll 1 \end{equation} 
    
    \item[\textit{iv)}] \emph{\textbf{HTP2}} The $d\times d$ matrix \begin{multline} 
    E^N_{\Omega, \e}(\theta)\equiv D\expan {K_\e} N (\t+\omega)^\top J\circ \expan {K_\e} N(\t+\omega) D\expan {K_\e} N (\t+\omega) \\ - D(f_{\e,\expan{\mu_\e} N}\circ            \expan {K_\e} N(\t))^\top J\circ            (f_{\expan {\e,\mu_\e} N}\circ              \expan {K_\e} N(\t)) D(f_{\e,\expan{\mu_\e} N}\circ \expan {K_\e} N(\t))
\end{multline}
is a trigonometric  polynomial of degree less than $aN$.
\end{itemize} 

Then 
\begin{equation}\label{order_R}
    \expan {R_\e} N\sim \O(|\e|^{N+1})
\end{equation}
 and for any $0<\delta\leq\r$ we have 
\begin{equation}\label{R}
\norm {\expan R N}{\r-\d}{\g_N}\leq C\nu^{-1}(aN)^\tau\delta^{-(d+1)}\norm {E^N}\r{\g_N}
\end{equation}
where $C=C(d,\norm {D\expan K N}\r{\g_N}, \norm {\expan \N N}\r{\g_N}, \norm {J\circ \expan K N}\r{\g_N} )$.
\end{lemma}
\begin{proof}

Writing  $\expan {R_\e} N$ in terms of $\expan{K_\e }{N}$  as in  \eqref{red_error} yields
$$\expan {R_\e} N(\t)=\left[DE^N_\e(\t)\bmid  \expan V N_\e (\t +\omega)\left( B_\e(\t)-\l(\e)\Id\right) + D\expan K N_\e(\t+\omega)\left(\tilde{S}_\e(\t
)-\expan S N_\e(\t)\right)\right]$$
with
\begin{align}
    \expan V N_\e (\t) &\equiv J^{-1}\circ	\expan {K_\e} N(\t) D\expan {K_\e} N(\t)\expan {\N_\e} N(\t) \\
    B_\e(\t)-\l(\e)\Id &\equiv -E^N_{L,\e}(\t +\omega)\expan {S_\e} N(\t) \label{B} \\
    \tilde{S}_\e(\t)-\expan S N_\e(\t) &\equiv -\expan \N N_\e(\t +\omega)^\top \expan\Gamma N_\e(\t+\omega)\expan \N N_\e(\t+\omega)\left( B_\e(\t)-\l(\e)\Id\right) \label{est-S}
\end{align}
where 
\begin{equation}\label{error_L}
E^N_{L,\e}(\t)\equiv D\expan K N_\e(\t)^\top J\circ \expan K N_\e(\t) D\expan K N_\e(\t)
\end{equation} 
is the pull back $(\expan{K_\e}{N})^*\Omega$ written in coordinates and $\expan\Gamma N_\e$ as in \eqref{Gamma_e}. We recall that $J$ is the matrix associated to the symplectic form, see Section \ref{Prelims}. It is easy to estimate the first column of $\expan{R_\e}{N}$ using Cauchy estimates, that is $$\normm{DE_\e^N}{\r-\d} \leq C\d^{-1} \normm{E_\e^N}{\r}$$

To obtain estimates for the second column of $\expan {R_\e} N$, due to \eqref{B} and \eqref{est-S}, it is enough to get estimates of $E^N_L$ . The estimate for $E^N_L$ is obtained using that $f_{\e,\mu}^*\Omega =\l(\e)\Omega$. Note that $E_{\Omega,\e}^N = (\expan{K_\e}{N}\circ T_\omega)^*\Omega -(f_{\e,\expan{\mu_\e}N}\circ \expan{K_\e}N)^*\Omega $ in coordinates and, since $(f_{\e,\expan{\mu_\e}N}\circ \expan{K_\e}N)^*\Omega = \l(\expan{K_\e}{N})^*\Omega $, we have that $E^N_L$ satisfies the equality

\begin{equation} \label{e_L_equation}
E_{L,\e}^N\circ T_\omega-\l(\e)E_{L,\e}^N=E_{\Omega, \e}^N.
\end{equation}
Then, by  Lemma \ref{lin-est} and \textbf{HTP2} we obtain
\begin{equation}\label{lagran_error}
\norm {E_L^N}{\r-\d}{\g_N} \leq C\nu^{-1}(aN)^\tau \delta^{-d} \norm{E_\Omega^N}{\r-\d/2}{\g_N}. 
\end{equation}

To get estimates for $E_\Omega^N$, we follow \cite{Cal-Cel-Lla-13}.
If $h$ and $g$ are smooth maps with range in $\M$, the matrix corresponding to $h^*\Omega -g^*\Omega$ is $$ Dh^\top J\circ h Dh - Dg^\top J\circ g Dg = (Dh^\top - Dg^\top)J\circ h Dh - Dg^\top( J\circ h - J\circ g) Dh +Dg^\top J\circ g(Dh- Dg) $$
Using this formula with $g=f_{\e,\expan{\mu_\e }{N}}\circ \expan{K_\e }{N}$, $h= \expan{K_\e }{N}\circ T_\omega$ and Cauchy estimates one obtains
\begin{equation}\label{omega_error}
 \normm {E^N_{\Omega,\e}}{\r-\d/2}\leq C\delta^{-1}\normm {E_\e^N}\r \end{equation}
which yields $ E^N_{L,\e}, E^N_{\Omega,\e}\sim\ord{N+1}$ and, then, $\expan {R_\e} N\sim\ord{N+1}$ and 
\begin{equation}
    \norm {\expan R N}{\r-\d}{\g_N}\leq C\nu^{-1}(aN)^\tau\delta^{-(d+1)}\norm {E^N}\r{\g_N}.\label{bound_lemma}
\end{equation}
Note that when the matrix $J$ is constant both \textbf{HTP2} and the computations above are significantly simpler than in the general case.
\end{proof}

\begin{remark}
We emphasize that, if $K_0$ satisfies $K_0\circ T_\omega - f_{0, \mu_0}\circ K_0=0$ then $DK_0(\t)^\top J\circ K_0 DK_0(\t) =0$ and  $K_0(\torus^d)$ is a Langrangian manifold, see \cite{Cal-Cel-Lla-13}. This implies that the spaces Range$( DK_0(\t))$ and Range$(J^{-1}\circ K_0(\t)DK_0(\t))$ are transversal and this condition makes $M_0(\t)$ a linear isomorphism. Note that if $E^N_L$ in \eqref{error_L} represents the error of the lagrangian character of $\expan{K_\e}{N}$, then, if $E_L^N$ is small enough the spaces Range$(D\expan{K_\e }{N}(\t))$ and Range$(J^{-1}\circ \expan{K_\e }{N}(\t)D\expan{K_\e }{N}(\t))$ will be transversal and the matrix $\expan{M_\e }{N}$ will define a linear isomorphism. This transversality will be obtained if \eqref{small_cond} is satisfied and it is given by \eqref{lagran_error} and \eqref{omega_error}.
\end{remark}

\subsection{Estimates for the corrections}
In this sections we obtain estimates for the corrections $\expanin \Delta N {2N}$ and $\expanin \sigma N{2N}$, this estimates are obtained by following the steps in Algorithm \ref{algo}. First, Lemma \ref{lem_w_sigma_estimates}, we obtain estimates for the corrections $\Delta_\e$, $\sigma_\e$ and then, using Cauchy estimates, we obtain estimates for the truncations $\expanin \Delta N{2N}$, $\expanin \sigma N{2N}$ , Corollary \ref{lem_est_trunc_ser}. \\

Consider $\C\subseteq \complex ^d/\integer ^d\times \complex ^d$ the  complexification of $\M=\torus ^d\times B$.

\begin{lemma}\label{lem_w_sigma_estimates}
Let $a\in\nat$, $0<\r<1$, and $\d$ such that $0<2\d<\r$. Assume that for any $\e\in \complex$, such that $|\e|<\gamma_N$, $f_{\e,\expan\mu N_\e}:\C\rightarrow\C$ is an analytic conformally symplectic map with $f_{\e,\expan\mu N_\e}^*\Omega=\l(\e)\Omega$. Assume also that $\expan K N\in\Aa\r{\g_N}$ is such that  $\expan K N_\e:\torus_\r^d\rightarrow\complex^d/\integer ^d\times\complex^d$ is an embedding. Assume also that for any $|\e|< \g_N$ we have the following:

\begin{itemize}
    \item[\textit{i)}] $\expan {K_\e} N\left(\torus_\r^d\right)\subset \domain (f_{\e,\expan \mu N})$ and that there exist $\xi\geq 0$ such that $$\dist\left(\expan {K_\e} N\left(\torus_\r^d\right),\partial\domain (f_{\e,\expan \mu N})\right)\geq \xi>0$$
    $$\dist\left(\expan{\mu_\e}{N}, \partial \Lambda\right)\geq \xi $$
    \item[\textit{ii)}] \textbf{\emph{HND}}.  The following non-degeneracy condition holds: $$\det\left(\begin{matrix} \overline{\expan {S_\e} N} & \, &  \overline{\expan {S_\e} N \noaverage {B_{b,\e}}} + \overline{\tilde{A}_{\e,1}^N}  \\  \e^3\Id & \, & \overline{\tilde{A}_{\e,2}^N} \end{matrix}\right)\neq 0$$
    \item[\textit{iii)}] For any $N\in\nat$, the matrices $\noaverage{\expanin{\tilde{E}_{\e,2}}N{2N}}$ and $\noaverage{\tilde{A}_{\e,2}^N}$ defined in \eqref{etilde} and \eqref{atilde}, are trigonometric polynomials of degree less or equal than $aN$.
\end{itemize}
Then, for any $0<r<1$ we have
\begin{equation}\label{W_and_sigma_order}
W_\e\sim\ord {N+1},\quad \sigma_\e\sim\ord{N+1} 
\end{equation}
\begin{equation}\label{w_final_estimate}
    \norm W{\r-\d}{r\g_N} \leq C\nu^{-3}(aN)^{2\tau}\d^{-(\tau +3d)}\frac{r^{N+1}}{1-r} \norm{E^N}\r{\g_N}
\end{equation}
and
\begin{equation}\label{sigma_est_fin}
    \sup_{|\e|\leq r\g_N}|\sigma_\e| \leq C\nu^{-1}(aN)^\tau\d^{-d}\frac{r^{N+1}}{1-r}\norm{E^N}\r{\g_N}
\end{equation}
where $C=C(d,\norm {D\expan K N}\r{\g_N}, \norm {\expan M N}\r{\g_N}, \norm {(\expan M N)^{-1}}\r{\g_N}, \norm {\expan\N N}\r{\g_N}, \T^N)$ and $\T^N$ is defined in \eqref{T}.
\end{lemma}
\begin{proof}
Given that $\noaverage{\expanin{\tilde{E}_{\e,2}}N{2N}}$ and $\noaverage{\tilde{A}_{\e,2}^N}$ are trigonometric polynomials, by Lemma \ref{lin-est}, \eqref{ba}, and \eqref{bb}; $B_a$ and $B_b$ satisfy the following estimates
\begin{eqnarray}
    \norm{B_a}{\r-\d}{r\g_N} &\leq C\nu^{-1}(aN)^\tau \d^{-d} \norm{\expanin {\tilde{E}_2} N{2N}}{\r}{r\g_N}\nonumber \nonumber\\
		& \leq C\nu^{-1}(aN)^\tau \d^{-d} \norm{\expanin {E} N{2N}}{\r}{r\g_N} \label{B_a_est} 
\end{eqnarray}
and similarly
\begin{equation}\label{B_b_est}
    \norm{B_b}{\r-\d}{r\g_N}\leq C\nu^{-1}(aN)^\tau\d^{-d} \norm{A^N}\r{r\g_N}.
\end{equation}
Taking into account that $W_2=\noaverage{W_2}+\overline{W_2}$ and $\noaverage{W_2} =\noaverage{B_a} +\sigma\noaverage{B_b}$, to have estimates for $W_2$ we need estimates for $\overline{W_2}$ and $\sigma$. Now, according to \eqref{w2barsigma} we have
\begin{equation}\label{inverse_eq}
    \left(\begin{matrix} \overline{W_{\e,2}} \\ \sigma_\e \end{matrix}\right) =
        \left(\begin{matrix} \overline{\expan {S_\e} N} & \, & \overline{\expan {S_\e} N \noaverage {B_{b,\e}}} + \overline{\tilde{A}_{\e,1}^N}  \\  \e^3\Id &\, & \overline{\tilde{A}_{\e,2}^N} \end{matrix}\right)^{-1}  \left(\begin{matrix} -\overline{\expan {S_\e} N \noaverage{B_{a,\e}}} - \overline{\expanin{\tilde{E}_{\e,1}}N{2N}}  \\  -\overline{\expanin{\tilde{E}_{\e,2}}N{2N}} \end{matrix}\right),
\end{equation}
denoting
\begin{equation}\label{T}
\T_\e^N:=\left\| \left(\begin{matrix} \overline{\expan {S_\e} N} & \, & \overline{\expan {S_\e} N \noaverage {B_{b,\e}} } + \overline{\tilde{A}_{\e,1}^N}  \\  \e^3\Id &\, & \overline{\tilde{A}_{\e,2}^N} \end{matrix}\right)^{-1}\right\| \qquad \mbox{and} \qquad \T^N= \sup_{|\e|\leq r\g_N} \T_\e^N 
\end{equation}
from \eqref{inverse_eq} we have 
\begin{equation}\label{w2eq}
    |\sigma_\e|, \left|\overline{W_{\e,2}}\right|\leq \T_\e^N \left(\left| \overline{\expan {S_\e} N \noaverage {B_{a,\e}}} + \overline{\expanin{\tilde{E}_{\e,1}}{N}{2N}}\right| +\left| \overline{\expanin{\tilde{E}_{\e,2}}N{2N}}\right|     \right)\sim \ord{N+1}  
\end{equation}
which yields $\sigma_\e \sim \ord{N+1}$ and $\overline{W_{\e,2}}\sim \ord{N+1}$
because $\noaverage{B_{a,\e}}\sim\ord{N+1}$ and $\expanin{\tilde{E}_\e}N{2N} \sim\ord{N+1}$. 

Thus
\begin{equation*}
\begin{split}
    |\sigma_\e|, |\overline{W_{\e,2}}| &\leq \T_\e^N \left(\left| \overline{\expan {S_\e} N \noaverage {B_{a,\e}}}\right| +\left| \overline{\expanin{\tilde{E}_{\e,1}}{N}{2N}}\right| +\left| \overline{\expanin{\tilde{E}_{\e,2}}N{2N}}\right|     \right)\\ 
    		& \leq C\T^N \left(\normm{ \expan {S_\e} N}{\r} \normm{\noaverage {B_{a,\e}}}{\r-\d} + \normm{\expanin{\tilde{E}_{\e,1}}{N}{2N}}{\r} + \normm{\expanin{\tilde{E}_{\e,2}}{N}{2N}}{\r}\right) 
\end{split}
\end{equation*}
for any $0< \delta <\r$. Thus, using \eqref{etilde} and \eqref{B_a_est}  we obtain
\begin{equation}\label{w2_bar_estimate}
    \sup_{|\e|\leq r\g_N}\left|\overline{W_{\e,2}}\right| \leq C\nu^{-1}(aN)^\tau\d^{-d} \norm{\expanin E N{2N}}\r{r\g_N} 
\end{equation}
\begin{equation}\label{sigma_estimate}
    \sup_{|\e|\leq r\g_N}\left|\sigma_\e\right| \leq C\nu^{-1}(aN)^\tau \d^{-d} \norm {\expanin E N{2N}}\r{r\g_N}.
\end{equation}
For $\noaverage {W_2}= \noaverage{B_a} +\sigma\noaverage{B_b}$ we have 
\begin{equation}\label{w2_average_est}
\begin{split}
    \norm{\noaverage {W_2}}{\r-\d}{r\g_N} &\leq \norm{\noaverage{B_a}}{\r-\d}{r\g_N} + \sup_{|\e|\leq r\g_N}\left|\sigma\right| \norm{\noaverage{B_b}}{\r-\d}{r\g_N} \\
		&\leq C\nu^{-1}(aN)^\tau\d^{-d}\norm{\expanin E N{2N}}\r{r\g_N} +C\nu^{-2}(aN)^{2\tau}\d^{-2d}\norm{A^N}\r{r\g_N} \norm{\expanin E N{2N}}\r{r\g_N}, \\
		& \leq C\nu^{-2}(aN)^{2\tau}\d^{-2d} \norm{\expanin E N{2N}}\r{r\g_N}.
\end{split}
\end{equation}
Thus, combining \eqref{w2_bar_estimate} and \eqref{w2_average_est} we get
\begin{equation}\label{w2_estimate}
    \norm{W_2}{\r-\d}{r\g_N} \leq   C\nu^{-2} (aN)^{2\tau}\d^{-2d}\norm{\expanin E N{2N}}\r{r\g_N}
\end{equation}
The estimates for $\noaverage{W_1} $ come from \eqref{w1} and Lemma \ref{classic_lin_est}, i.e., 
\begin{multline*}
\norm{\noaverage{W_1}}{\r-2\d}{r\g_N}\\     \leq C\nu^{-1}\d^{-(\tau +d)}\left[\norm{\expan S N}{\r-\d}{r\g_N} \norm{W_2}{\r-\d}{r\g_N} +\norm{\expanin {\tilde{E}}N{2N}}{\r-\d}{r\g_N} + \sup_{|\e|\leq r\g_N}\left|\sigma_\e\right| \norm{\tilde{A}^N}{\r-\d}{r\g_N} \right] \end{multline*}
\begin{align*}  \leq & C\nu^{-1}\d^{-(\tau + d)}\left[ \norm{\expan S N}\r{r\g_N} \nu^{-2}(aN)^{2\tau}\d^{-2d} \norm {\expanin E N{2N}}\r{r\g_N}\right. \\
      & \left. +\norm{\left(\expan M N \right)^{-1}}\r{r\g_N} \norm{\expanin E N{2N}}\r{r\g_N} +\norm{A^N}\r{r\g_N} \nu^{-1}(aN)^\tau\r^{-d}\norm{\expanin E N{2N}}\r{r\g_N} \right]
\end{align*}
that is,
\begin{equation}\label{w1_estimate}
    \norm{\noaverage{W_1}}{\r-2\d}{r\g_N}\leq C\nu^{-3}(aN)^{2\tau}\d^{-(\tau +3d)}\norm{\expanin E N{2N}}\r{r\g_N}.
\end{equation}
Finally, the estimate for $\overline{W_1}$ comes from \eqref{uni-cond}, that is 
\begin{eqnarray}
    \sup_{|\e|\leq r\g_N}\left| \overline{W_{\e,1}}\right| &\leq C\left( \norm{\noaverage{W_1}}{\r-\d}{r\g_N} + \norm{W_2}{\r-\d}{r\g_N}\right) \nonumber\\
	&\leq C\nu^{-3}(aN)^{2\tau}\d^{-(\tau +3d)}\norm{\expanin E N{2N}}\r{r\g_N} \label{w1_bar_est}.
\end{eqnarray}
Putting together \eqref{w2_estimate}, \eqref{w1_estimate}, \eqref{w1_bar_est}, and using the Cauchy estimates in Corollary \ref{lem_cauchy_est_ser} yields the claimed estimate for $W$.
\end{proof}

\begin{cor}\label{lem_est_trunc_ser}
Assuming the hypothesis of Lemma \ref{lem_estimate_reducibility} and Lemma \ref{lem_w_sigma_estimates}, for any $0<\d<\r$ and $0<r<1$ we have
\begin{equation}
    \norm{\expanin \Delta N {2N}}{\r-\d}{r\g_N}  \leq C\nu^{-3}(aN)^{2\tau}\delta^{-(\tau + 3d)} \frac{r^{N+1}}{(1-r^{1/2})^2}\norm{E^N}{\r}{\g_N} 
\end{equation}
\begin{equation}
    \sup_{|\e|\leq r\g_N}\left|\expanin	{\sigma_\e} N{2N}\right|  \leq C\nu^{-1}(aN)^\tau\d^{-d} \frac{r^{N+1}}{(1-r^{1/2})^2}\norm{E^N}{\r}{\g_N}
\end{equation}
Moreover,
\begin{equation}
    \norm{\expanin \Delta  {2N}\infty}{\r-\d}{r\g_N} \leq C\nu^{-3}(aN)^{2\tau}\delta^{-(\tau + 3d)} \frac{r^{\frac{3}{2} N+1}}{(1-r^{1/2})^2}\norm{E^N}{\r}{\g_N}
\end{equation}
\begin{equation}
    \sup_{|\e|\leq r\g_N}\left|\expanin	{\sigma_\e} {2N}\infty \right| \leq C\nu^{-1}(aN)^\tau\d^{-d} \frac{r^{\frac{3}{2} N+1}}{(1-r^{1/2})^2}\norm{E^N}{\r}{\g_N}
\end{equation}
\end{cor}

\begin{proof}
Using the Cauchy estimates as in Corollary \ref{lem_cauchy_est_ser} and the estimates in Lemma \ref{lem_w_sigma_estimates} one obtains
\begin{align*}
\norm {\expanin \Delta {2N}\infty}{\r-\d}{r^2\g_N} & \leq \frac{r^{2N+1}}{(1-r)} \norm \Delta{\r-\d}{r\g_N} \\
	& \leq C\frac{r^{2N+1}}{1-r}\nu^{-3}(aN)^{2\tau}\d^{-(\tau +3d)}\frac{r^{N+1}}{1-r} \norm{E^N}{\r}{\g_N} \\
	& = C\frac{r^{3N+2}}{(1-r)^2}\nu^{-3}(aN)^{2\tau}\d^{-(\tau +3d)} \norm{E^N}{\r}{\g_N}
\end{align*}
and
\begin{align*}
\sup_{|\e|\leq r^2\g_N}\left| \expanin {\sigma_\e}{2N}\infty\right| & \leq \frac{r^{2N+1}}{1-r}\sup_{|\e|\leq r\g_N}\left| \sigma_\e\right| \\
 	&\leq \frac{r^{2N+1}}{1-r} C\nu^{-1}(aN)^\tau\d^{-d}\frac{r^{N+1}}{1-r} \norm{E^N}{\r}{\g_N} \\
 	& = C\nu^{-1}(aN)^\tau\d^{-d}\frac{r^{3N+2}}{(1-r)^2} \norm{E^N}{\r}{\g_N}
\end{align*}
The other estimates are obtained similarly.
\end{proof}

\subsection{Non-linear estimates for the quasi-Newton method.}
The quasi-Newton procedure in Algorithm \ref{algo} can also be described using a convenient operator notation. Defining the error functional 
\begin{equation}\label{def_func_error}
    \E[K_\e,\mu_\e] = f_{\e,\mu_\e}\circ K_\e - K_\e\circ T_\omega
    \end{equation}
and assuming $\Delta$ and $\sigma$ are \emph{small} enough, the Taylor expansion of $\E[ K  +\Delta,\mu  +\sigma]$ is given by 
\begin{equation}\label{taylor_exp} 
    \E[ K  +\Delta, \mu  +\sigma]=  \E[ K ,\mu ] +D_1\E[ K,\mu ]\Delta +D_2\E[ K ,\mu ]\sigma + \R[\Delta,\sigma; K , \mu ] 
\end{equation}
where the Frechet derivatives are given by
\begin{align}
    D_1\E[ K_\e ,\mu_\e ]\Delta_\e &= \left(Df_{\e, \mu_\e}\circ K_\e\right)\Delta_\e -\Delta_\e\circ T_\omega \\ 
    D_2\E[ K_\e ,\mu_\e ]\sigma_\e &= \left(D_\mu f_{\e, \mu_\e }\circ K_\e\right)\sigma_\e 
    \end{align}
and $\R$ is the remainder of the Taylor expansion. Note that $\E[\expan{K_\e}N, \expan{\mu_\e}N] = E^N_\e$, with this notation the \emph{classic} Newton method would consist in finding a correction $(\expanin{\Delta_\e}N{2N}, \expanin{\mu_\e }N{2N}) $ such that  \begin{equation}  
\E[\expan{K_\e}{N} ,\expan{\mu_\e}{N} ]+ D_1\E[ \expan{K_\e}{N} ,\expan{\mu_\e}{N} ]\expanin{\Delta_\e}N{2N} +D_2\E[ \expan{K_\e}{N} ,\expan{\mu_\e}{N} ]\expanin{\sigma_\e}N{2N}=0. \label{clas_new_met_func}
\end{equation}
As it was explained before, in Section \ref{newton_method}, the corrections we construct with Algorithm \ref{algo} do not satisfy \eqref{clas_new_met_func} but they solve an approximate equation \eqref{quasi_newton}. The following Lemmas give estimates for the error functional evaluated in the corrected unknowns. First, Lemma \ref{lem_tayl_expan_1}, we give estimates for the error $\E[\expan K N +\Delta, \expan \mu N +\sigma]$ and then, using Cauchy estimates, we obtain the estimates for the error evaluated in the truncated corrections, $ \E[\expan K N +\expanin\Delta N{2N}, \expan \mu N + \expanin \sigma N{2N}] $, Proposition \ref{lem_taylor_est}. 

\begin{remark}
We emphasize that to be able to compute $\E[K+\Delta, \mu+\sigma]$ we need both $\Delta$ and $\sigma$ to be \emph{small enough}, so the compositions in \eqref{def_func_error} are well defined. In particular $\Delta$ and $\sigma$ need to satisfy $\| \Delta\| , | \sigma | \leq \xi$ and we need to choose the domain loss. In Section \ref{sec_proof_main}, Lemma \ref{induction}, we give smallness conditions on the initial error which will guarantee that the compositions will be defined at any step of the iteration.  This is very standard in KAM theory.
\end{remark}

\begin{lemma}\label{lem_est_lin_part}
Assume $0< r<1$ and $0<\delta\leq \r$. Then, under the hypothesis of  Lemma \ref{lem_estimate_reducibility} and Lemma \ref{lem_w_sigma_estimates} one has 

\begin{equation}\label{lin_order_in_eps}
\E[\expan {K_\e} N,\expan{\mu_\e} N] +D_1\E[\expan {K_\e} N,\expan{\mu_\e} N]\Delta_\e +D_2\E[\expan {K_\e} N,\expan{\mu_\e} N]\sigma_\e \sim\ord{2N+1}
\end{equation}
and
\begin{multline}\label{est_lin_part}
\norm{\,\E[\expan K N,\expan{\mu} N] +D_1\E[\expan {K} N,\expan\mu N] \Delta +D_2\E[\expan K N,\expan\mu N]\sigma\,}{\r-\d}{r\g_N} \\
\leq \frac{r^{2N+1}}{1-r}\norm{E^N}\r{\g_N} + C\nu^{-4}(aN)^{3\tau}\d^{-(\tau +4d +1)}\frac{r^{N+1}}{1-r}\norm{E^N}\r{\g_N} ^2 \\
\end{multline}
\end{lemma}
\begin{proof}
Note that with the operator notation introduced at the beginning of this section we  have  $ \E(\expan K N,\expan\mu N)=E^N$. Using \eqref{red_error_eq_1} and taking into account that $\Delta_\e= \expan {M_\e} N W_\e$ and that $W_\e$ satisfies \eqref{quasi_newton} we have 
\begin{align}
  &\E[\expan {K_\e} N,\expan{\mu_\e} N] +D_1\E[\expan {K_\e} N,\expan{\mu_\e} N]\Delta_\e +D_2\E[\expan {K_\e} N,\expan{\mu_\e} N]\sigma_\e \nonumber \\
 	&\quad=E_\e^N +\left(Df_{\e,\expan \mu N}\circ \expan{K_\e}{N}\right)\Delta_\e -\Delta_\e \circ T_\omega + \left(D_\mu f_{\e,\expan \mu N}\circ \expan{K_\e }{N}\right)\sigma_\e -\expan {R_\e} N\left(\expan {M_\e} N\right)^{-1}\Delta_\e  \nonumber \\
	& \quad\quad+\expan {R_\e} N\left(\expan {M_\e} N\right)^{-1}\Delta_\e \nonumber  \\
	&\quad=E_\e^N + \expan {M_\e} N\circ T_\omega\left(\begin{matrix} \Id & \expan {S_\e} N \\ 0 & \l(\e)\Id\end{matrix}\right)\left(\expan {M_\e} N\right)^{-1}\Delta_\e -\Delta_\e\circ T_\omega +\left(D_\mu f_{\e, \expan\mu N}\circ \expan {K_\e} N\right)\sigma_\e \nonumber \\
	&\quad	\quad+\expan {R_\e} N \left(\expan {M_\e} N\right)^{-1}\Delta_\e \nonumber \\ 
    &\quad =E_\e^N -\expanin {E_\e} N{2N} + \expan {R_\e} N W_\e  \\
	&\quad = \expanin {E_\e}{2N}\infty + \expan {R_\e} N W_\e\sim \ord{2N+1} \nonumber
\end{align}
where $\expanin {E_\e} {2N}{\infty}= \sum_{n=2N+1}^{\infty} E_n\e^n$. Note that the order of $\e$ in the last line follows from the definition of $\expanin E {2N}\infty$, \eqref{order_R}, and \eqref{W_and_sigma_order}. 

Then, using the Cauchy estimates of Corollary \ref{lem_cauchy_est_ser}, Lemma \ref{lem_estimate_reducibility}, and Lemma \ref{lem_w_sigma_estimates} one obtains
\begin{align*}
    &\norm{\E[\expan {K} N,\expan{\mu} N] +D_1\E[\expan {K} N,\expan{\mu} N]\Delta +D_2\E[\expan {K} N,\expan{\mu} N]\sigma}{\r-\d}{r\g_N}\\
    & \qquad \leq \norm{\expanin E{2N}\infty}{\r-\d}{r\g_N} +\norm {\expan RN}{\r-\d}{r\g_N}\norm W{\r-\d}{r\g_N}  \\
	& \qquad \leq \frac{r^{2N+1}}{1-r}\norm{E^N}\r{\g_N} + C\nu^{-4}(aN)^{3\tau}\d^{-(\tau +4d +1)}\frac{r^{N+1}}{1-r}\norm{E^N}\r{\g_N}^2
\end{align*}
\end{proof}
\begin{lemma}\label{lem_tayl_expan_1}
Assume $0< r<1$ and $0<\delta\leq \r$. Then, under the hypothesis of Lemma \ref{lem_w_sigma_estimates} and Lemma \ref{lem_estimate_reducibility} we have 
\begin{equation}\label{order_error_corrected_1}
\E(\expan {K_\e} N +\Delta_\e, \expan {\mu_\e} N+\sigma_\e)\sim\ord{2N+1}
\end{equation}
and
\begin{equation}
\norm{\E[\expan K N +\Delta, \expan \mu N+\sigma]}{\r-\d}{r\g_N} 
\leq \frac{r^{2N+1}}{1-r}\norm{E^ N}\r{\g_N} +C\nu^{-6}(aN)^{4\tau}\d^{-(2\tau +6d)}\frac{r^{N+1}}{1-r} \norm{E^N}\r{\g_N}^2
\end{equation}
where $C=C\left(\norm{D\expan{K}{N}}{\r }{\g_N}, \norm{D^2f_{\expan{\mu}{N}}\circ \expan{K}{N}}{\r}{\gamma_N}, \norm{D_\mu^2f_{\expan{\mu}{N}}\circ \expan{K}{N}}{\r}{\gamma_N} \right)$. 

\end{lemma}

\begin{proof}
Note that $\R[\expan {K_\e}N,\expan{\mu_\e} N,\Delta_\e,\sigma_\e]$ in \eqref{taylor_exp} can be estimated using Taylor estimates for the remainder, that is
    \begin{equation}\label{taylor_error_1}
    \normm {\R_\e}{\r}  \leq C\left(\normm{\Delta_\e}{\r}^2 +|\sigma_\e|^2 \right)
    \end{equation}
where $C$ is a constant depending on the norms of the second derivatives of $f_{\e,\mu}$ evaluated at $\expan{K_\e }{N}$ and $\expan{\mu_\e }{N}$. \\
Since $f_{\e,\mu}$ is assumed to be analytic it is natural to expect the quantities $\norm{D^2f_{\expan{\mu}{N}}\circ \expan{K}{N}}{\r}{\gamma_N}$, $ \norm{D_\mu^2f_{\expan{\mu}{N}}\circ \expan{K}{N}}{\r}{\gamma_N}$ to be close to $\norm{D^2f_{\expan{\mu}{N_0}}\circ \expan{K}{N_0}}{\r_0}{\gamma_{N_0}}$, $ \norm{D_\mu^2f_{\expan{\mu}{N_0}}\circ \expan{K}{N_0}}{\r_0}{\gamma_{N_0}}$, at the first step of the iterations. For now, we assume that $C$ is uniform constant. In Section \ref{sec_proof_main}, Lemma \ref{induction}, we give sufficient conditions on the initial error of the iteration that imply that $C$ can be taken as an uniform constant during all the iterations.

Note that \eqref{taylor_error_1} yields $\R_\e\sim\ord{2N+2}$. This, together with \eqref{lin_order_in_eps}, gives \eqref{order_error_corrected_1}. Moreover, taking sup with respect to $\e$ one obtains
\begin{align*}
\norm \R{\r-\d}{r\g_N} & \leq C\left(\norm\Delta{\r-\d}{r\g_N}^2 +\sup_{|\e|\leq r\g_N}|\sigma|^2 \right) \\
	& \leq C\left(\norm{\expan M N}\r{\g_N}^2 \norm W{\r-\d}{r\g_N} + \sup_{|\e|\leq r\g_N}|\sigma|^2 \right) \\
	& \leq C \left(\nu^{-6}(aN)^{4\tau}\d^{-(2\tau +6d)}\frac{r^{2N+2}}{(1-r)^2}\norm{E^N}\r{r\g_N} \right. + \left. \nu^{-2}(aN)^{2\tau}\d^{-2d}\frac{r^{2N+N}}{(1-r)^2} \norm{E^N}\r{r\g_N}\right) \\
	 & \leq C\nu^{-6}(aN)^{4\tau}\d^{-(2\tau +6d)}\frac{r^{2N+2}}{(1-r)^2}  \norm{E^N}\r{r\g_N}^2
\end{align*}
where in the third line we use the inequalities in Lemma \ref{lem_w_sigma_estimates}. Finally, this inequality, Lemma \ref{lem_est_lin_part}, and \eqref{taylor_exp} give the result.
\end{proof}

Note that the estimates above are done for the analytic functions $\Delta$ and $\sigma $. It is only left to get the respective estimates for the truncations $\expanin{\Delta}{N}{2N}$ and $\expanin{\sigma}{N}{2N}$, which are an easy consequence of the Cauchy inequalities and are given in the following propositions.

\begin{proposition}\label{lem_lin_est_trunc}
Assuming the hypothesis of Lemma \ref{lem_estimate_reducibility} and Lemma \ref{lem_w_sigma_estimates}, for any $0<\d<\r$ and $0<r<1$ we have
\begin{equation} \label{order_lin_part}
    \E[\expan {K_\e}N,\expan{\mu_\e} N] + D_1\E[\expan {K_\e}N,\expan{\mu_\e} N]\expanin{\Delta_\e} N{2N} + D_2\E[\expan {K_\e}N,\expan{\mu_\e} N]\expanin{\sigma_\e} N{2N} \sim \ord{2N+1}
\end{equation} and
\begin{align}
     &\norm{\E[\expan {K}N,\expan{\mu} N] + D_1\E[\expan {K}N,\expan\mu N]\expanin\Delta N{2N} + D_2\E[\expan KN,\expan\mu N]\expanin\sigma N{2N}}{\r-\d}{r\g_N}\nonumber \\
	& \qquad\qquad\leq C\nu^{-3}(aN)^{2\tau}\d^{-(\tau +3d)}\frac{r^{\frac{3}{2}N+1}}{(1-r^{1/2})^2} \norm{E^N}\r{\g_N} + C\nu^{-4}(aN)^{3\tau}\d^{-(\tau +4d +1)}\frac{r^{N+1}}{1-r} \norm{E^N}\r{\g_N}^2 \label{lin_est_N_2N}
\end{align}
\end{proposition}
\begin{proof}
Recalling the notation $\expanin{\Delta_\e} a\infty \equiv \sum_{n=a+1}^\infty \Delta_n(\t)\e^n$ we have that $\dpy{\expanin \Delta N{2N}}+\dpy{\expanin \Delta {2N}\infty =\Delta}$. Also remember that $E^N=\E[\expan KN,\expan\mu N]$, then, using the linearity of the Frechet derivatives one obtains  
\begin{align*}
 &\E[\expan {K_\e}N,\expan{\mu_\e} N] + D_1\E[\expan {K_\e}N,\expan{\mu_\e} N]\expanin{\Delta_\e} N{2N} + D_2\E[\expan {K_\e}N,\expan{\mu_\e} N]\expanin{\sigma_\e} N{2N} \\
	& \qquad = \E[\expan {K_\e}N,\expan{\mu_\e} N] +D_1\E[\expan {K_\e}N,\expan{\mu_\e} N] \Delta_\e  + D_2\E[\expan {K_\e}N,\expan{\mu_\e} N] \sigma_\e  \\
	& \qquad \qquad - D_1\E[\expan {K_\e}N,\expan{\mu_\e} N]\expanin{\Delta_\e} {2N}\infty  - D_2\E[\expan {K_\e}N,\expan{\mu_\e} N]\expanin  {\sigma_\e}{2N}\infty \\
	& \qquad = \E[\expan {K_\e}N,\expan{\mu_\e} N] +D_1\E[\expan {K_\e}N,\expan{\mu_\e} N] \Delta_\e  + D_2\E[\expan {K_\e}N,\expan{\mu_\e} N] \sigma_\e   \\
	& \qquad\qquad -\left(Df_{\e,\expan\mu N}\circ \expan {K_\e} N\right)\expanin{\Delta_\e}{2N}\infty + \expanin{\Delta_\e}{2N}\infty\circ T_\omega - \left(D_\mu f_{\e,\expan\mu N}\circ \expan {K_\e} N\right)\expanin{\sigma_\e}{2N}\infty 
\end{align*} which implies \eqref{order_lin_part}. Moreover, using the relation above and the estimates in Lemma \ref{lem_est_lin_part} and Lemma \ref{lem_est_trunc_ser} one gets
\begin{align*}
 &\norm{\E[\expan {K}N,\expan{\mu} N] + D_1\E[\expan {K}N,\expan\mu N]\expanin\Delta N{2N} + D_2\E[\expan KN,\expan\mu N]\expanin\sigma N{2N}}{\r-\d}{r\g_N} \\
	& \qquad \leq \norm{\E[\expan K N,\expan{\mu} N] +D_1\E[\expan {K} N,\expan\mu N] \Delta +D_2\E[\expan K N,\expan\mu N]\sigma}{\r-\d}{r\g_N} \\
	& \qquad\quad  + C ( \norm{\expanin \Delta{2N}\infty}{\r-\d}{r\g_N} + \sup_{|\e|\leq r\g_N}\left|\expanin{\sigma_\e}{2N}\infty\right|) \\
	&\qquad \leq \frac{r^{2N+1}}{1-r}\norm{E^N}\r{\g_N} + C\nu^{-4}(aN)^{3\tau}\d^{-(\tau +4d +1)}\frac{r^{N+1}}{1-r}\norm{E^N}\r{\g_N} ^2 \\
	&\qquad\quad + C\nu^{-3}(aN)^{2\tau}\d^{-(\tau +3d)}\frac{r^{\frac{3}{2}N+1}}{(1-r^{1/2})^2}\norm{E^N}\r{\g_N}  + C\nu^{-1}(aN)^{\tau}\r^{-d}\frac{r^{\frac{3}{2} N+1}}{(1-r^{1/2})^2}\norm{E^N}\r{\g_N}  \\
	&\qquad \leq C\nu^{-3}(aN)^{2\tau}\d^{-(\tau +3d)}\frac{r^{\frac{3}{2}N+1}}{(1-r^{1/2})^2} \norm{E^N}\r{\g_N} + C\nu^{-4}(aN)^{3\tau}\d^{-(\tau +4d +1)}\frac{r^{N+1}}{1-r} \norm{E^N}\r{\g_N}^2
\end{align*}
\end{proof}

\begin{proposition}\label{lem_taylor_est}
Assuming the hypothesis of Lemma \ref{lem_estimate_reducibility} and Lemma \ref{lem_w_sigma_estimates}, for any $0<\d<\r$ and $0<r<1$ we have
\begin{equation}\label{ord_error_corrected}
\E\left[\expan {K_\e}N+ \expanin{\Delta_\e} N{2N}, \expan{\mu_\e} N + \expanin {\sigma_\e} N{2N}\right]\sim\ord{2N+1}
\end{equation}
and 	
\begin{align}
  &\norm{\E[\expan KN+ \expanin\Delta N{2N}, \expan\mu N + \expanin \sigma N{2N}]}{\r-\d}{r\g_N} \label{taylor_est}\\
	& \qquad \leq C\nu^{-3}(aN)^{2\tau}\d^{-(\tau +3d)}\frac{r^{\frac{3}{2}N+1}}{(1-r^{1/2})^2} \norm{E^N }\r{\g_N}\nonumber + C\nu^{-6}(aN)^{4\tau}\d^{-(2\tau +6d)}\frac{r^{N+1}}{(1-r^{1/2})^4} \norm{E^N }\r{\g_N}^2\nonumber \end{align}
where $C=C(d,\norm {\expan MN}{\r}{\g_N}$, $\norm {\left(\expan MN\right)^{-1}}{\r}{\g_N}$, $\norm {\expan \N N}{\r}{\g_N}$, $\norm {D\expan KN}{\r}{\g_N}$,$\T)$, the constant $C$ also depends on the norms of the  first and second derivatives of $f_{\e,\mu}$ evaluated at $\expan{K_\e}{N}$ and $\expan{\mu_\e}{N}$.
\end{proposition}

\begin{proof}
The expansion \eqref{ord_error_corrected} follows from using  the same argument as in the proof of Lemma \ref{lem_tayl_expan_1}. We also have
\begin{align*}
  &\norm{\R\left[\expan KN,\expan\mu N,\expanin \Delta N{2N}, \expanin\sigma N{2N}\right]}{\r-\d}{r\g_N} \\
	&\qquad \leq C\left( \norm {\expanin\Delta N{2N}}{\r-\d}{r\g_N}^2 +\sup_{|\e|\leq r\g_N} \left|\expanin {\sigma_\e} N{2N}\right|^2\right) \\
	&\qquad \leq C\left(\nu^{-6}(aN)^{4\tau}\d^{-(2\tau +6d)}\frac{r^{2N+2}}{(1-r^{1/2})^4} \norm {E^N}{\r-\d}{r\g_N}^2 + \nu^{-2}(aN)^{2\tau}\r^{-2d}\frac{r^{2N+2}}{(1-r^{1/2})^4} \norm {E^N}{\r-\d}{r\g_N}^2 \right).
\end{align*}
Combining this estimate with \eqref{lin_est_N_2N} in Lemma \ref{lem_lin_est_trunc} one gets \eqref{taylor_est}.
\end{proof}

\section{Iteration of the quasi-Newton method.} \label{sec_proof_main}

We start this section giving the choice of parameters which quantify the loss of regularity at any step of the quasi Newton method. Lemma \ref{induction} will  guarantee  that the Newton method is well defined at any step. We note that we have loss of domain in both the variable on the torus, $\t$, and the variable of the perturbation, $\e$. In contrast with the regular KAM theory we end up losing much more domain in $\e$, so that at the end we do not have any $\e$ domain. 

\subsection{The iterative procedure.}
We denote by $h\in\nat$ the number of steps of the quasi Newton method. We consider
\begin{equation}\label{delta_h}
\d_h:=\frac{\r_0}{2^{h+2}}\quad\mbox{and}\quad \r_{h+1}:=\r_{h} -\d_h \geq \frac{\r_0}{2}\quad\mbox{for }h\geq 1,
\end{equation}
where $\r_h$ denotes the radius of analyticity in the variable $\t$ at step $h$, that is, at step $h$ we will be considering functions in the space $\A_{\r_h}$. Note that $\r_0=\r'$ can be the one given in Theorem \ref{theo_prev_paper}. Since at any step we double the number of coefficients of the Lindstedt expansions, we have, 
\begin{equation}\label{N_h}
N_h:= 2^hN_0
\end{equation}
and
\begin{equation}\label{gamma_h}
\quad \tilde{\g}_h:=\g_{N_h}= \left(\frac{\nu}{2}\right)^{1/\alpha} \frac{1}{(aN_h)^{\tau/\alpha}} = \left(\frac{\nu}{2}\right)^{1/\alpha} \frac{1}{(a2^hN_0)^{\tau/\alpha}}
\end{equation}
where $\alpha\in\nat$ is the exponent in $\l(\e)=1-\e^\alpha$, $a\in\nat$, and $N_0\in \nat$ is a fixed constant to be chosen later. Note that $\tilde{\g}_h$ is the radius of the domain of analyticity in the variable $\e$ at step $h$, that is, at step $h$ we will be considering functions in the space $\A_{\r_h,\tilde{\g}_h}$. Also note that
\begin{equation}\label{gamma_h_2}
    \tilde{\g}_{h+1} =2^{-\tau/\alpha} \tilde{\g}_h.
\end{equation}
Denoting $K_0 :=\expan K {N_0}$ and $\mu_0:= \expan\mu {N_0}$,
for $h \geq 1$ we have
\begin{equation}
    K_h:= \expan K{N_0} +\expanin \Delta {N_0}{N_1} + \cdots +\expanin \Delta {N_{h-1}}{N_h} \qquad \mu_h := \expan\mu {N_0} +\expanin\sigma{N_0}{N_1}+ \cdots +\expanin\sigma{N_{h-1}}{N_h}.
\end{equation}
Furthermore, denoting
\begin{equation}
\Delta_h := \expanin\Delta{N_h}{N_{h+1}}\quad\mbox{and}\quad \sigma_h :=\expanin \sigma{N_h}{N_{h+1}}\quad\mbox{for }h\geq 0 
\end{equation}
we have that, for $h\geq 0$
\begin{equation}
    K_{h+1} =K_h +\Delta_h	\quad\mbox{and}\quad \mu_{h+1}=\mu_h +\sigma_h. 
\end{equation}
Finally, denote also 
\begin{align}
 & e_h:=\norm {\E[K_h,\mu_h]}{\r_h}{\tilde{\g}_h} =\norm{E^{N_h}}{\r_h}{\tilde{\g}_h} \\
 & d_h:=\norm {\Delta_h}{\r_{h+1}}{\tilde{\g}_{h+1}} \\
 & v_h:=\norm {D\Delta_h}{\r_{h+1}}{\tilde{\g}_{h+1}} \\
 & s_h:= \sup_{|\e |\leq\tilde{\g}_{h+1}} |\sigma_h(\e) |.
\end{align}
\begin{remark}
We emphasize the dependence of $\tilde{\g}_h$ in $N_h$, note that $\tilde{\g}_h \rightarrow 0 $ as $N_h \rightarrow \infty$ ($h\rightarrow\infty$). This implies that this quasi Newton method will not converge in any Banach space $\A_{\r_h,\tilde{\g}_h}$, because the domains in $\e$ shrink to $0$, however, at each step we get estimates in balls with positive radius, $\tilde{\g}_h$. An analysis of these bounds will provide us with estimates of the coefficients of the expansion.
Note also that to start with $e_0\ll 1$ we require $N_0$ sufficiently large in the formal power series in Theorem \ref{theo_prev_paper}.   
\end{remark}
Note that with this new notation the estimates in Corollary \ref{lem_est_trunc_ser} can be written as
\begin{align}
     & d_h \leq \hat{C}_h \nu^{-3}(aN_h)^{2\tau}\d_h^{-(\tau +3d)}\left(\frac{1}{2^{\tau/\alpha}} \right)^{N_h}e_h \label{d_h}\\
     & v_h \leq \hat{C}_h \nu^{-3}(aN_h)^{2\tau}\d_h^{-(\tau +3d +1)}\left(\frac{1}{2^{\tau/\alpha}} \right)^{N_h}e_h \\
    & s_h \leq \hat{C}_h \nu^{-1}(aN_h)^\tau \d_h^{-d} \left(\frac{1}{2^{\tau/\alpha}}\right)^{N_h}e_h
\end{align}
where $\hat{C}_h$ is an explicit constant depending in a polynomial manner on $\norm {M_h}{\r_h}{\tilde{\g}_h}$, $\norm {M_h^{-1}}{\r_h}{\tilde{\g}_h}$, $\norm {\N_h}{\r_h}{\tilde{\g}_h}$, $\norm {DK_h}{\r_h}{\tilde{\g}_h}$, and $\T_h$. 
Moreover, the non linear estimate \eqref{taylor_est} given in Proposition \ref{lem_taylor_est} implies 
\begin{equation}\label{taylor_est_h}
e_{h+1}\leq \tilde{C}_h \nu^{-6} (aN_h)^{4\tau}\d_h^{-(2\tau +6d)}\left(\frac{1}{2^{\tau/\alpha}}\right)^{N_h} \left( e_h + e_h^2 \right)
\end{equation}
where $\tilde{C}_h$ is a constant which also depends explicitly on $\norm {M_h}{\r_h}{\tilde{\g}_h}$, $\norm {M_h^{-1}}{\r_h}{\tilde{\g}_h}$, $\norm {\N_h}{\r_h}{\tilde{\g}_h}$, $\norm {DK_h}{\r_h}{\tilde{\g}_h}$, and $\T_h$.

\begin{remark}\label{rem_constants}
In the following we will denote $C$ a constant depending on $\nu, \tau, d, \xi, \r_0, \left|J^{-1}\right|$; and that is a polynomial in $\norm {M_0}{\r_0}{\tilde{\g}_0}$, $\norm {M_0^{-1}}{\r_0}{\tilde{\g}_0}$, $\norm {\N_0}{\r_0}{\tilde{\g}_0}$, $\norm {DK_0}{\r_0}{\tilde{\g}_0}$, and $\T_0$. We will also denote $$C_h = \max\left( \hat{C}_h, \tilde{C}_h \right).$$  In Lemma \ref{induction} , we  give smallness conditions so that $C_h\leq C$ for every $h\geq 0$. Since we are working  with expansions near to $(\expan K{N_0}, \expan\mu {N_0})$ it is natural to expect that the quantities $\norm {M_h}{\r_h}{\tilde{\g}_h}$, $\norm {M_h^{-1}}{\r_h}{\tilde{\g}_h}$, $\norm {\N_h}{\r_h}{\tilde{\g}_h}$, $\norm {DK_h}{\r_h}{\tilde{\g}_h}$, and $\T_h$ will be close to $\norm {M_0}{\r_0}{\tilde{\g}_0}$, $\norm {M_0^{-1}}{\r_0}{\tilde{\g}_0}$, $\norm {\N_0}{\r_0}{\tilde{\g}_0}$, $\norm {DK_0}{\r_0}{\tilde{\g}_0}$, and $\T_0$, respectively. For now, we assume that $C$ is large enough, for instance $C> 2C_0$.  Here $M_h=\expan M{N_h}$, $\N_h= \expan \N {N_h}$, and $\T_h= \T^{N_h}$  as in \eqref{M}, \eqref{N}, and \eqref{T}.
\end{remark}

Considering this uniform constant $C$ on \eqref{taylor_est_h}, and  taking $N_0$ sufficiently large, yields $e_h<1$  for any $h>0$, and  inequality \eqref{taylor_est_h} implies
\begin{equation}\label{error_h_est}
e_{h+1} \leq C \nu^{-6} (aN_h)^{4\tau} \delta_h^{-(2\tau +6d)} \left(\frac{1}{2^{\tau/\alpha}}\right)^{N_h} e_h.
\end{equation}

\begin{remark}\label{rem_taylor_est_simp}
Due to Remark \ref{rem_constants} and the definitions of $\d_h, \r_h, N_h$, and $\tilde{\g}_h$; the inequality \eqref{error_h_est} can be rewritten as $$e_{h+1}\leq C \nu^{-6} (aN_0)^{4\tau}\r_0^{-(2\tau +6d)} 2^{-(4\tau +12d)}\left(2^h\right)^{6\tau +6d} \left(\frac{1}{2^{\tau/\alpha}}\right)^{2^hN_0}e_h$$ or 
\begin{equation}\label{taylor_h}
    e_{h+1}\leq CDB^h r^{2^hN_0}e_h
\end{equation}
where 
$$D=\nu^{-6}(aN_0)^{4\tau}\r_0^{-(2\tau +6d)} 2^{-(4\tau +12d)},\quad r= 2^{-\tau/\alpha}\quad \mbox{and}\quad B= 2^{6\tau +6d}.$$
\end{remark}

\begin{lemma}\label{induction} Assuming  that $2^{3(\tau +3d)+1}CDBr^{N_0}\leq\frac{1}{2}$, $Br^{N_0} <1$, $ N_0^{2\tau}e_0 \ll 1$, and $$C\nu^{-3}(aN_0)^{2\tau}\r_0^{-(\tau +3d+1)} 2^{2\tau +6d +2} e_0\ll 1.$$
Then, for all integers $h\geq 0$ the following properties hold:
\begin{itemize}
    \item[$(p1;h)$] $$\norm {K_h-K_0}{\r_h}{\tilde{\g}_h}\leq \ell_K N_0^{2\tau}e_0 < \xi$$ $$\sup_{|\e |\leq \tilde{\g}_{h+1}} \left|\mu_h -\mu_0\right| \leq \ell_\mu N_0^\tau e_0 <\xi$$ with $\ell_K\equiv  C\nu^{-3}a^{2\tau}\r_0^{-(\tau +3d)} 2^{2\tau +6d} $ and $\ell_\mu \equiv C\nu^{-1} a^\tau 2^d\r_0^{-d}$
    \item[$(p2;h)$]
    $$e_h\leq  (CD)^h B^{h^2} r^{(2^h-1)N_0}e_0$$
    \item[$(p3;h)$]
    $$C_h\leq C$$
\end{itemize}
\end{lemma}

\begin{remark}
Note that by \eqref{norm_E^N} we have $e_0 \sim \mathcal{O}(N_0^{-(\tau/\alpha)N_0})$, due to the fact that we estimate $e_0$ in a ball with radius $\tilde{\g}_0 \sim \mathcal{O}(N_0^{-\tau/\alpha}) $. So the assumptions on the smallness of $N_0e_0$ are satisfied. 
\end{remark}

\begin{proof}
Note that $(p1;0)$, $(p2;0)$, and $(p3;0)$ are trivial. 

Let us now prove $(p1,H+1)$, $(p2,H+1)$, and $(p3,H+1)$ assuming they are true for $h=1,2,...,H$. Noticing that $2^j\leq 2^{j+1}-1$, for any $j\geq 0$,  and assuming that $N_0$ is large enough such that $2^{3(d+\tau)}CDBr^{N_0}\leq\frac{1}{2}$ and $Br^{N_0} <1$,  we have  
\begin{align*}
    \norm {K_{H+1}-K_0}{\r_{H+1}}{\tilde{\g}_{H+1}} & = \norm{\expanin\Delta{N_0}{N_1}+ ... + \expanin\Delta{N_H}{N_{H+1}}}{\r_{H+1}}{\tilde{\g}_{H+1}} \\
	    & \leq \sum_{j=0}^H d_j \leq \sum_{j=0}^H \hat{C}_j \nu^{-3}(aN_j)^{2\tau}\d_j^{-(\tau+ 3d)} r^{N_j}e_j \\
	    & \leq \sum_{j=0}^H C\nu^{-3}(a2^jN_0)^{2\tau}\r_0^{-(\tau +3d)} 2^{2\tau +6d}2^{(\tau +3d)j} r^{2^jN_0}e_j \\
	    & \leq C\nu^{-3} (aN_0)^{2\tau}\r_0^{-(\tau +3d)} 2^{2\tau +6d} \sum_{j=0}^H 2^{3(d+\tau)j} r^{2^jN_0} \left((CD)^j B^{j^2}r^{(2^j-1)N_0}e_0\right) \\
	    & \leq C\nu^{-3} (aN_0)^{2\tau}\r_0^{-(\tau +3d)} 2^{2\tau +6d}  \sum_{j=0}^H 2^{3(d+\tau)j} (CD)^j B^{j^2}r^{(2^{j+1}-1)N_0}e_0 \\
	    & \leq C\nu^{-3} (aN_0)^{2\tau}\r_0^{-(\tau +3d)} 2^{2\tau +6d} \sum_{j=0}^H 2^{3(d+\tau)j} (CD)^j B^{j^2}r^{2^jN_0}e_0 \\
	    & \leq C\nu^{-3} (aN_0)^{2\tau}\r_0^{-(\tau +3d)} 2^{2\tau +6d} e_0 \sum_{j=0}^H \left( 2^{3(d+\tau)}CDBr^{N_0} \right)^j \\
	    & \leq C\nu^{-3}(aN_0)^{2\tau}\r_0^{-(\tau +3d)} 2^{2\tau +6d} e_0 \\
	    & \leq \ell_K N_0^{2\tau}e_0
\end{align*}
Similarly,
\begin{align*}
\sup_{|\e |\leq \tilde{\g}_{H+1}} \left|\mu_{H+1} -\mu_0\right| & =\sup_{|\e |\leq \tilde{\g}_{H+1}} \left| \expanin\sigma{N_0}{N_1} + ... + \expanin\sigma{N_H}{N_{H+1}} \right| \\
	& \leq \sum_{j=0}^H s_j \leq \sum_{j=0}^H \hat{C}_j\nu^{-1} (aN_j)^\tau \d_j^{-d} r^{N_j}e_j\\
	& \leq \sum_{j=0}^H C\nu^{-1}(a2^jN_0)^\tau \r_0^{-d}2^{(j+2)d} r^{2^jN_0} \left((CD)^j B^{j^2}r^{(2^j-1)N_0}e_0\right) \\
	& \leq C\nu^{-1} (aN_0)^\tau \r_0^{-d}2^{2d} \sum_{j=0}^H (2^{\tau+d})^j (CD)^j B^{j^2}r^{(2^{j+1}-1)N_0} e_0 \\
	& \leq C\nu^{-1} (aN_0)^\tau \r_0^{-d}2^{2d} \sum_{j=0}^H (2^{\tau+d})^j (CD)^j B^{j^2}r^{2^jN_0} e_0 \\
	& \leq  C\nu^{-1} (aN_0)^\tau \r_0^{-d}2^{2d}e_0 \sum_{j=0}^H\left(2^{\tau+d} CDBr^{N_0}\right)^j \\
	& \leq  C\nu^{-1}(aN_0)^\tau 2^d\r_0^{-d}e_0 \\
	& \leq \ell_\mu N_0^\tau e_0.
\end{align*} 
Thus, taking  $N_0$ large  enough, which makes $e_0$ small, we get  $\ell_K N_0^{2\tau}e_0<\xi$ and $\ell_\mu N_0^{\tau}e_0<\xi$.

Since $(p_1;H+1)$ is true, we use the estimate \eqref{taylor_h} given in  Remark \ref{rem_taylor_est_simp}, which is a consequence of the nonlinear  estimates given in Lemma \ref{lem_taylor_est},  that is 
\begin{equation}
e_{h+1}=\norm {\E(K_{h} +\Delta_{h}, \mu_{h} +\sigma_{h})}{\r_{h+1}}{\tilde{\g}_{h+1}} \leq CDB^{h} r^{2^{h}N_0}e_{h}
\end{equation}
where $D$, $B$, and $r$ are as in Remark \ref{rem_taylor_est_simp}. This yields,

\begin{align*}
e_{h+1} & \leq CDB^{h}r^{2^{h}N_0}e_{h} \\
	& \leq CDB^{h}r^{2^{h}N_0}\left((CD)^h B^{h^2}r^{(2^{h}-1)N_0}e_{0}\right) \\
	& \leq (CD)^{h+1} B^{h^2+h}r^{(2^{h+1}-1)N_0} e_{0} \\
	& \leq (CD)^{h+1} B^{(h+1)^2}r^{(2^{h+1}-1)N_0} e_{0}
\end{align*}
which yields $(p_2,H+1)$.
In order to prove $(p_3;H+1)$  note that 
\begin{eqnarray}
\norm{\N_h - \N_0}{\r_h}{\tilde{\g}_h} \leq \overline{C} \norm {DK_h -DK_0}{\r_h}{\tilde{\g}_h} \\
\norm{M_h - M_0}{\r_h}{\tilde{\g}_h} \leq \overline{C} \norm {DK_h -DK_0}{\r_h}{\tilde{\g}_h} \\
\norm{M_h^{-1} - M_0^{-1}}{\r_h}{\tilde{\g}_h} \leq \overline{C} \norm {DK_h -DK_0}{\r_h}{\tilde{\g}_h} \\
|\T_h-\T_0| \leq \overline{C} \norm {DK_h -DK_0}{\r_h}{\tilde{\g}_h} 
\end{eqnarray}
where $\overline{C}$ is a uniform constant. The above inequalities come from the fact that $M_h$, $\N_h$, and $\T_h$ are algebraic expressions of $DK_h$, $Df_{\cdot,\mu_h}$, and $D_\mu f_{\cdot,\mu_h}$; see \eqref{M}, \eqref{N}, \eqref{S}, \eqref{T}. Then,

\begin{align*}
\norm{DK_{H+1} -DK_0}{\r_{H+1}}{\tilde{\g}_{H+1}} & = \norm{\expanin {D\Delta}{N_0}{N_1}+ ... + \expanin{D\Delta}{N_H}{N_{H+1}}}{\r_{H+1}}{\tilde{\g}_{H+1}} \\
	& \leq \sum_{j=0}^H d_j \leq \sum_{j=0}^H \hat{C}_j \nu^{-3}(aN_j)^{2\tau}\d_j^{-(\tau+ 3d+1)} r^{N_j}e_j \\
	& \leq \sum_{j=0}^H C\nu^{-3}(a2^jN_0)^{2\tau}\r_0^{-(\tau +3d +1)} 2^{2\tau +6d+2}2^{(\tau +3d+1)j} r^{2^jN_0}e_j \\
	& \leq C\nu^{-3} (aN_0)^{2\tau}\r_0^{-(\tau +3d +1)} 2^{2\tau +6d +2} \sum_{j=0}^H 2^{(3d+3\tau +1)j} r^{2^jN_0} \left((CD)^j B^{j^2}r^{(2^j-1)N_0}e_0\right) \\
	& \leq C\nu^{-3} (aN_0)^{2\tau}\r_0^{-(\tau +3d +1)} 2^{2\tau +6d +2}  \sum_{j=0}^H 2^{(3d +3\tau +1)j} (CD)^j B^{j^2}r^{(2^{j+1}-1)N_0}e_0 \\
	& \leq C\nu^{-3} (aN_0)^{2\tau}\r_0^{-(\tau +3d +1)} 2^{2\tau +6d +2}  \sum_{j=0}^H 2^{(3d+3\tau +1)j} (CD)^j B^{j^2}r^{2^{j}N_0}e_0 \\
	& \leq C\nu^{-3} (aN_0)^{2\tau}\r_0^{-(\tau +3d +1)} 2^{2\tau +6d +2} e_0 \sum_{j=0}^H \left( 2^{3d+3\tau+1}CDBr^{N_0}\right)^j \\
	& \leq C\nu^{-3}(aN_0)^{2\tau}\r_0^{-(\tau +3d+1)} 2^{2\tau +6d +2}e_0 \\
\end{align*}
where the sum is bounded as in the previous estimates. Taking $e_0$ small enough, such that $\overline{C}C\nu^{-3}(aN_0)^{2\tau}\r_0^{-(\tau +3d+1)} 2^{2\tau +6d +2} e_0 \ll 1$, we are able to verify $(p3;H+1)$ because $C_{H+1}$ is an algebraic expression of $M_H$, $\N_H$, and $\T_H$; and taking $C\geq 2C_0$, for example.
\end{proof}

\subsection{Proof of main Lemma \ref{main_theo}.}\label{proof-main-lemma}
For the proof of the main Lemma we inherit all the notation introduced throughout this section.

\begin{proof}
Note that Theorem \ref{theo_prev_paper}  assures the existence of the Lindstedt series satisfying \eqref{norm_error_1}. That is, given $K_0\in \A_\r$ and $\mu_0\in\Lambda\subseteq \complex$ satisfying $f_{0,\mu_0}\circ K_0 = K_0\circ T_\omega$ and \textbf{HND},  there exists $\r_0<\r$ and power expansions $\expan{K_\e }{N}$ and $\expan{\mu_\e}{N}$ such that $$\label{norm_error_1}
\normm{f_{\e,\expan {\mu_\e} N}\circ\expan {K_\e}N -\expan {K_\e}N\circ T_\omega}{\r'}\leq C_N|\e|^{N+1}$$ for any $N\geq0$. This expansion is unique under the normalization condition \eqref{normal_coeff}.

If $\expan{K_\e }{N}$ and $\expan{\mu_\e}{N}$ satisfy hypothesis \textbf{HTP1} and \textbf{HTP2} then, we can choose $N_0$ such that $\expan K{N_0}$ and $\expan{\mu }{N_0} $ satisfy  the hypothesis  of Lemmas \ref{lem_estimate_reducibility} and \ref{lem_w_sigma_estimates}. Also, $N_0$ needs to be large enough such that $2^{3(\tau +3d)+1}CDBr^{N_0}\leq\frac{1}{2}$, $Br^{N_0} <1$, $\ell_K N_0^{2\tau}e_0 < \xi$, $\ell_\mu N_0^\tau e_0<\xi$ and $$\overline{C}C\nu^{-3}(aN_0)^{2\tau}\r_0^{-(\tau +3d+1)} 2^{2\tau +6d +2} e_0\ll 1,$$ then  Lemma \ref{induction} can be applied and this allows us to iterate the quasi Newton method described in Algorithm \ref{algo}. That is, we can construct the unique formal power series as follows $$\expan{K_\e}{N_0} +\expanin{\Delta_\e}{N_0}{2N_0} + \expanin{\Delta_\e}{2N_0}{2^2N_0} + \cdots + \expanin{\Delta_\e}{2^hN_0}{2^{h+1}N_0} + \cdots  $$
$$\expan{\mu_\e}{N_0} +\expanin{\mu_\e}{N_0}{2N_0} + \expanin{\mu_\e}{2N_0}{2^2N_0} + \cdots + \expanin{\mu_\e}{2^hN_0}{2^{h+1}N_0} + \cdots  $$
Note that by definition of $\tilde{\g}_h$ we will have $\tilde{\g}_h= r^h\tilde{\g}_0$ , where $r=2^{-\tau/\alpha}$ and $\tilde{\g}_0=2^{-1/\alpha}\nu^{1/\alpha}(aN_0)^{-\tau /\alpha}$, see \eqref{gamma_h_2}. 
Before giving the detailed computations, note that $\tilde{\g}_h \sim (2^hN_0)^{-\tau/\alpha}$ and if $n\in\left(2^hN_0, 2^{h+1}N_0\right]\cap\nat$ then $$(\tilde{\g}_h)^{-n} \sim (2^hN_0)^{C(\tau/\alpha)2^hN_0} \sim n^{C(\tau/\alpha)n}.$$ Using this together with Cauchy estimates is expected to yield the Gevrey estimates.   
More precisely, if $n\in\left(2^hN_0, 2^{h+1}N_0\right]\cap\nat$, using Cauchy estimates, \eqref{d_h}, and $(p2;h)$ we have
\begin{align*}
\normm{K_n}{\frac{\r_0}{2}} &\leq \left( \tilde{\g}_{h+1}\right)^{-n}\norm {\Delta_h}{\frac{\r_0}{2}}{\tilde{\g}_{h+1}} \\
    &\leq \left( \tilde{\g}_{h+1}\right)^{-n}\norm {\Delta_h}{\r_{h+1}}{\tilde{\g}_{h+1}}
    \\
	& \leq  (r^{h+1}\tilde{\g}_0)^{-n} d_h 
	\\
	& \leq (r^{h+1}\tilde{\g}_0)^{-2^{h+1}N_0} \hat{C}_h \nu^{-3}(aN_h)^{2\tau}\d_h^{-(\tau +3d)}r ^{N_h}e_h \\
	& \leq (r^{h+1}\tilde{\g}_0)^{-2^{h+1}N_0} C\nu^{-3} \left(a2^hN_0 \right)^{2\tau} \r_0^{-(\tau +3d)} 2^{(2\tau +6d)} 2^{(\tau +3d)h} r^{2^hN_0}(CD)^h B^{h^2} r^{(2^h-1)N_0}e_0 
	\\
	& \leq C\nu^{-3}\r_0^{-(\tau +3d)} 2^{(2\tau +6d)}(aN_0)^{2\tau}e_0 (2^{3\tau +3d}CD)^hB^{h^2}  (\tilde{\g}_0)^{-2^{h+1}N_0}  r^{(-(h+1)2^{h+1} +2^{h+1}-1)N_0} 
	\\
	& \leq C\nu^{-3}\r_0^{-(\tau +3d)} 2^{(2\tau +6d)}(aN_0)^{2\tau}e_0 (2^{3\tau +3d}CD)^hB^{h^2} (2^{1/\alpha}\nu ^{-1/\alpha} (aN_0)^{\tau/\alpha})^{2^{h+1}N_0} r^{-(h2^{h+1}+1)N_0} 
	\\
	& \leq \hat{L} \left(2^{3\tau +3d}CDB 2^{2/\alpha}\nu ^{-2/\alpha} a^{2\tau/\alpha}\right)^{2^h N_0} (N_0^{2\tau/\alpha})^{2^hN_0} (2^{\tau/\alpha})^{(h2^{h+1}+1)N_0} 
	\\
	& \leq \hat{L}2^{(\tau/\alpha)N_0} F^{2^hN_0} (N_0^{2\tau/\alpha})^{2^hN_0} (2^{2\tau/\alpha})^{h2^hN_0} 
	\\
	& \leq L F^{2^hN_0}(2^hN_0)^{(2\tau/\alpha)2^hN_0}  
	\\
    & \leq LF^n n^{(2\tau/\alpha)n}
    \end{align*}
where $\hat{L}=C\nu^{-3}\r_0^{-(\tau +3d)} 2^{2\tau +6d}(aN_0)^{2\tau}e_0$, $F=2^{3\tau +3d+2/\alpha}CDB\nu ^{-2/\alpha} a^{2\tau/\alpha}$, and $L=\hat{L}(2^{\tau/\alpha})^{N_0}$. The estimates for $\mu_n$ are obtained in a similar way.
\end{proof}

\bigskip




\subsection{Proof of Theorem~\ref{gev-asymptotics}}\label{proof-gev-asymptotics}
\begin{proof}
Inheriting the notation from Lema~\ref{induction}, consider $N_0$ sufficiently large such that the a-posteriori theorem, Theorem 14 in \cite{Cal-Cel-Lla-16}, can be applied. That is, $N_0$ such that 
\begin{equation}
\sup_{\e\in\G, |\e|\leq \tilde{\gamma}_0} \normm{E^{N_0}_\e}{\r} \leq \hat{C}\left(\nu\tilde{\nu}(\l;\omega,\tau) \right)^2\d^{-4(\tau+\delta)}.
\end{equation}
where $\tilde{\nu}(\l;\omega,\tau)$ is defined in \eqref{nu-lambda}.
Then, following the discussion in Section~\eqref{sec-gev-asymp} and applying the a-posteriori theorem, Theorem 14 in \cite{Cal-Cel-Lla-16}, one obtains 
\begin{align*}   
\sup_{\e\in\G, |\e|\leq \tilde{\gamma}_{h+2}} \normm{ \expan{K_\e}{2^hN_0} - K_\e }{\r_0-\d} & \leq \hat{C} \nu^{-1}\tilde{\nu}(\l;\omega,\tau)^{-1}\d^{-2(\tau+d)} \sup_{\e\in\G, |\e|\leq \tilde{\gamma}_{h+2}} \normm{E^{2^hN_0}_\e}{\r_0}
\end{align*}
where $\G$ is defined in \eqref{set-G}.

Now, considering $n\in(2^hN_0,2^{h+1}N_0]\cap\nat$ one has
\begin{align*}
\sup_{\e\in\G, |\e|\leq \tilde{\gamma}_{h+2}} \normm{ \expan{K_\e}{n} - K_\e }{\r_0-\d} & \leq \sup_{\e\in\G, |\e|\leq \tilde{\gamma}_{h+2}} \normm{ \expan{K_\e}{2^{h+1}N_0}  -\expanin{\Delta_\e}{n}{2^{h+1}N_0}  - K_\e }{\r_0-\d}   \\
	&\leq \sup_{\e\in\G, |\e|\leq \tilde{\gamma}_{h+2}} \normm{ \expan{K_\e}{2^{h+1}N_0} - K_\e }{\r_0-\d} +  \sup_{\e\in\G, |\e|\leq \tilde{\gamma}_{h+2}} \normm{\expanin{\Delta}{n}{2^{h+1}N_0}}{\r_0-\d}  \\
	& \leq \hat{C} \nu^{-1}\tilde{\nu}^{-1}\d^{-2(\tau+d)} \sup_{\e\in\G, |\e|\leq \tilde{\gamma}_{h+2}} \normm{E^{2^hN_0}_\e}{\r_0} +  \sup_{\e\in\G, |\e|\leq \tilde{\gamma}_{h+2}} \normm{\expanin{\Delta_\e}{n}{2^{h+1}N_0}}{\r_0-\d} \\
	&\leq \hat{C} \nu^{-1}\tilde{\nu}^{-1}\d^{-2(\tau+d)} \norm{E^{2^hN_0}_\e}{\r}{\tilde{\g}_{h+2}} +  \norm{\expanin{\Delta_\e}{n}{2^{h+1}N_0}}{\r_0-\d}{\tilde{\g}_{h+2}} \\
	&\leq \hat{C} \nu^{-1}\tilde{\nu}^{-1}\d^{-2(\tau+d)}  \norm{E^{2^hN_0}_\e}{\r_0}{\tilde{\g}_{h}} +   \frac{r^{n+1}}{1-r}\norm{\expanin{\Delta_\e}{2^hN_0}{2^{h+1}N_0}}{\r_0-\d}{\tilde{\g}_{h+1}} \\
	&\leq \hat{C} \nu^{-1}\tilde{\nu}^{-1}\d^{-2(\tau+d)} e_h +   r^{n+1}d_h \\ 
	&\leq \hat{C} \nu^{-1}\tilde{\nu}^{-1}\d^{-2(\tau+d)} e_h + r^{n+1} C \nu^{-3}(aN_h)^{2\tau}\d_h^{-(\tau +3d)}r^{N_h}e_h  \\
	&\leq \left( U + C \nu^{-3} (aN_0)^{2\tau}\r_0^{-(\tau+3d)} 2^{2\tau+6d}     2^{h(3\tau +3d)} r^{n+1}r^{2^hN_0}\right)(CD)^h B^{h^2}r^{(2^h-1)N_0}e_0    \\
	&\leq \left( U +  V  2^{h(3\tau +3d)} r^{n+1}r^{2^hN_0}\right)(CD)^h B^{h^2}r^{(2^h-1)N_0}e_0    \\
\end{align*}

where $U=\hat{C} \nu^{-1}\tilde{\nu}^{-1}\d^{-2(\tau+d)}$ and $V=C \nu^{-3} (aN_0)^{2\tau}\r_0^{-(\tau+3d)} 2^{2\tau+6d}$
\end{proof}

\medskip
\appendix
\section{ The case of the dissipative standard map of Theorem \ref{main_theo_dsm}  }\label{appendix}

\subsection{Verifying trigonometric polynomial hypothesis for the dissipative standard map}
Consider the dissipative standard map $f_{\e,\mu_\e}:\torus\times \real \rightarrow \torus\times \real$ given by 

\begin{equation} \label{dis-est-map-appendix}
f_{\e,\mu_\e}(x,y)= \left(x +\l(\e)y +\mu_\e - \e V(x),  \l(\e)y +\mu_\e - \e V(x) \right).
\end{equation}
Where $V(x)$ is a trigonometric polynomial. In this section we verify that maps like \eqref{dis-est-map-appendix} satisfy \textbf{HTP1} and \textbf{HTP2} of Lemma~\ref{main_theo}. For the sake of simplicity in the exposition we do it for the case $\l(\e)=1-\e^3$. The general case for $\alpha\in\nat$ is done by very similar computations, fixing the value of $\alpha=3$ allows an easy analysis of the Lindstedt series. 

Note that one has $f_{\e,\mu}^* \Omega = \l(\e)\Omega$ for the  symplectic form $\Omega_{(x,y)} = dx\wedge dy$, so it is conformally symplectic. One can write the map as 
\begin{align*}
x_{n+1} &= x_n +y_{n+1} \\
y_{n+1} &= \l(\e)y_n +\mu_\e -\e V(x_n)
\end{align*}  equivalently  
\begin{equation} x_{n+1} -(1+\l(\e)) x_n + \l(\e) x_{n-1} -\mu_\e +\e V(x_n)=0 \label{lagran_x_n}. 
\end{equation}
Considering a parametric representation of the variable $x_n\in \torus$ as $x_n=\t_n +u_\e(\t_n)$, $\t_n\in \torus$; where $u_\e: \torus \rightarrow \real$ is a $1$-periodic function and assuming that $\t_n$ varies linearly, i.e., $\t_{n+1} = \t_n +\omega$, then, \eqref{lagran_x_n} becomes \begin{equation}
u_\e(\t+\omega)-(1+\l(\e))u_\e(\t) +\l(\e)u_\e(\t-\omega) +(1-\l(\e))\omega -\mu_\e +\e V(\t +u_\e(\t)) =0 \label{u_eq}
\end{equation}
If  $u_\e$ satisfies \eqref{u_eq} it is easy to check that $K_\e: \torus \rightarrow \torus\times\real$, given by $$K_\e(\t)=\begin{pmatrix} \t +u_\e(\t) \\ \omega + u_\e(\t) - u_\e(\t-\omega) \end{pmatrix},$$ satisfies 
$f_{\e,\mu_\e}\circ K_\e(\t) = K(\t +\omega).$ Therefore, the problem of finding Lindstedt series for quasiperiodic orbits for the map $f_{\e,\mu_\e}$ is equivalent to find asymptotic power series to a solution, $(u_\e,\mu_\e)$, of \eqref{u_eq}. \\ 
Using $\l(\e) = 1-\e^3$, equation \eqref{u_eq} becomes
    \begin{equation}\label{eq_u} u_\e(\t+\omega)-(2-\e^3)u_\e(\t) +(1-\e^3)u_\e(\t-\omega) +\e^3\omega -\mu_\e + \e V(\t +u_\e(\t)) =0. 
    \end{equation}
Introducing the operator $$L_\omega u(\t) = u(\t +\omega) -2u(\t) +u(\t-\omega),$$ and expanding in power series on $\e$, i.e.,  $ u_\e(\t)=\sum_{n=0}^\infty u_n(\t)\e^n $   and $ \mu_\e =\sum_{n=0}^\infty \mu_n\e^n$  equation \eqref{eq_u} becomes \begin{multline}
\sum_{k=0}^2\left(L_\omega u_k(\t) - \mu_k\right)\e^k - \left(L_\omega u_3(\t) -\mu_3 +u_0(\t) -u_0(\t-\omega) - \omega\right)\e^3 \\
  + \sum_{k=4}^\infty \left( L_\omega u_k(\t) -\mu_k + u_{k-3}(\t) -u_{k-3}(\t-\omega)\right)\e^k = - \sum_{k=1}^\infty S_{k-1}(\t)\e^k \label{power_ser_eq} 
\end{multline} 


\begin{remark}
When $V(\t)$ is a trigonometric polynomial, the coefficients $S_n$ can be computed as follows. Note that $V_k(\t)=\hat{f}_k e^{2\pi i k\t}$ satisfies the relation 
\begin{equation}\label{deriv-trig}
\frac{d}{d\e}V_k(\t +u_\e(\t)) = 2\pi i k\frac{d}{d\e}u_\e(\t)V_k(\t +u_\e(\t)).
\end{equation} 

Thus, considering $$ V_k(\t+u_\e(\t)) =\sum_{n=0}^\infty S_n^k(\t)\e^n$$ and \eqref{deriv-trig} the coefficients $S_n^k$ satisfy the following relation
\begin{equation}\label{coef-trig-pol}
(n+1)S_{n+1}^k = \sum_{\ell =0}^n 2\pi ik (\ell +1)u_{\ell +1} S_{n-\ell}^k,
\end{equation}
and $S_0^k(\t)= \hat{f}_k e^{2\pi i k\t}$.
Furthermore, if $V(\t) = \sum_{|k|\leq a} \hat{f}_k e^{2\pi k \t} =\sum_{|k|\leq a} V_k(\t)$ is a trigonometric polynomial of degree $a$, considering $$ V(\t+u_\e(\t)) =\sum_{n=0}^\infty S_n(\t)\e^n,$$ the coefficients $S_n(\t)$ are given by $$ S_n(\t) = \sum_{|k|\leq a}S_n^k(\t) $$ where $S_n^k$ is given by \eqref{coef-trig-pol}.

\end{remark}

\begin{remark}
Note that if $\eta$ is a trigonometric polynomial and $\varphi$ is a solution of the equation $L_\omega \varphi =\eta$ then, $\varphi$ is a trigonometric polynomial of the same degree as $\eta$. This is due to the fact that the Fourier coefficients of $\varphi$ satisfy $\hat{\varphi}_k =\frac{1}{2(\cos (2\pi k\cdot \omega) -1)}\hat{\eta}_k $. Note that the equation $L_\omega \varphi = \eta$ has a solution if $\int_\torus \eta(\t)d\t =0$, and this solution is unique if we impose the normalization $\int_\torus \varphi(\t)d\t =0$.
\end{remark}

\begin{proposition} If $V(\t)$, in \eqref{dis-est-map-appendix}, is a trigonometric polynomial of degree $a$, then  $u_n(\t)$ is a trigonometric polynomial of degree $an$. Furthermore, $S_{n-1}(\t)$ is a trigonometric polynomial of degree $an$. 
\end{proposition}

\begin{proof}
Equating the terms of same order in equation \eqref{power_ser_eq} one gets that for order zero $\mu_0=0$ and $u_0(\t)\equiv 0$. For order 1 we have, $$L_\omega u_1(\t) -\mu_1 = -S_0(\t).$$ So, taking $\mu_1=0$, $u_1$ becomes a trigonometric polynomial of degree $a$, because $S_0(\t)=V(\t)$. Now, for order 2 we have $$L_\omega u_2(\t) -\mu_2 = - S_1(\t),$$ if $\mu_2=0$ the right hand side is   $S_1(\t) = \sum_{|k|\leq a} S_1^k(\t)=  2\pi i u_1(\t)\sum_{ |k|\leq a} k S_0^k(\t)$ which is a trigonometric polynomial of degree $2a$, thus $u_2$ is a trig polynomial of degree $2a$. For order three we have $$L_\omega u_3(\t) -\mu_3 +\omega = -S_2(\t),$$ here we take $\mu_3=\omega$ and $u_3$ is a trig polynomial of degree $3a$ because $$ S_2(\t)= \sum_{|k|\leq a} S_2^k(\t) =\pi i u_1(\t)\sum_{|k|\leq a} k S_1^k(\t) +2\pi iu_2(\t) \sum_{|k|\leq a} k S_0^k(\t) $$ is of degree $3a$; then $u_3(\t)$ is of degree $3a$. Finally, for $n\geq 4$, assume the claim is valid for any $m < n$ then, the equation of order $n$ is $$L_\omega u_n(\t) = \mu_n - u_{n-3}(\t) +u_{n-3}(\t-\omega)-  S_{n-1}(\t). $$ So, taking $\mu_n =\int_\torus S_{n-1}(\t)d\t$, $u_n$ can be found and has degree $an$ since, $S_{n-1}=\sum_{|k|\leq n} S_{n-1}^k$ and each $S_{n-1}^k$  has degree $an$ due to \eqref{coef-trig-pol}. Note $u_{n-3}$ has degree $(n-3)a$.
\end{proof}

\begin{cor} If $V(\t)$, in \eqref{dis-est-map-appendix}, is a trigonometric polynomial of degree $a$, then for any fixed $\e$ the sum $\displaystyle{\sum_{n=0}^N u_n(\t)\e^n}$ is a trig polynomial of degree $aN$ in $\t$.	
\end{cor}

Note that in this case \begin{equation}\label{explic_K}
\expan{ K_\e} N(\t) = \begin{pmatrix} \t +\sum_{n=0}^N u_n(\t)\e^n \\ \omega +\sum_{n=0}^N(u_n(\t)-u_n(\t-\omega))\e^n \end{pmatrix},
\end{equation} 
and using equation \eqref{power_ser_eq} we have 
\begin{align*}
    E_\e^N(\t) &:= f_{\e, \expan \mu N}\circ \expan {K_\e}N(\t) - \expan {K_\e}N(\t+\omega) =\sum_{n=N+1}^\infty \begin{pmatrix}
	S_{n-1}(\t) \\  S_{n-1}(\t)\end{pmatrix}\e^n 
\end{align*}
and therefore, for any fixed $\e$, $\expanin {E_\e} N {2N} (\t)$ 
is a trigonometric polynomial of degree $2aN$. Moreover, in this case the matrix $\expan {M_\e}N(\t) = \left[D\expan {K_\e}N (\t)| J^{-1}\circ \expan {K_\e}N(\t) D\expan {K_\e}N (\t)\expan {\N_\e} N(\t)\right]$ is given by
$$\expan {M_\e}N(\t) = \begin{bmatrix} 1+ \sum_{k=0}^N u_k'(\t)\e^k & \expan {\N_\e} N(\t)\sum_{k=0}^N (u'_k(\t-\omega)- u'_k(\t))\e^k  \\ \sum_{k=0}^N (u'_k(\t)- u'_k(\t -\omega))\e^k & \expan {\N_\e} N(\t) (1+ \sum_{k=0}^N u_k'(\t)\e^k)  \end{bmatrix}$$
where $\expan {\N_\e} N(\t)= \left((1+ \sum_{k=0}^N u_k'(\t)\e^k)^2 + (\sum_{k=0}^N (u'_k(\t)- u'_k(\t -\omega))\e^k)^2 \right)^{-1}$. So, $$\left(\expan {M_\e}N\circ T_\omega\right)^{-1} = \begin{bmatrix}  \left(\expan {\N_\e} N\circ T_\omega\right) (1+ \sum_{k=0}^N u_k'(\t+\omega)\e^k) & \left(\expan {\N_\e} N\circ T_\omega\right)\sum_{k=0}^N (u'_k(\t+\omega)- u'_k(\t))\e^k  \\ \sum_{k=0}^N (u'_k(\t)- u'_k(\t +\omega))\e^k &       1+ \sum_{k=0}^N u_k'(\t)\e^k \end{bmatrix}$$ 
which implies that $\expanin {\tilde{E}_{\e,2}} N{2N}$ is a trigonometric polynomial of degree $3aN$. Remember that $\expanin {\tilde{E}_{\e,2}} N{2N}$ is the second row of the vector $\expanin {\tilde{E}_\e} N{2N} = \left(\expan M N_\e\circ T_\omega\right)^{-1} \expanin {E_\e}N{2N}$. Note that $J=\begin{pmatrix}  0 & 1 \\ -1 &0\end{pmatrix}$.

Furthermore, we have $D_\mu f_{\e,\expan {\mu_\e} N}(x,y)= \begin{pmatrix} 1 \\ 1 \end{pmatrix}$, then the second row, $\tilde{A}_{\e,2}^N$, of the vector \\ $\tilde{A}_\e^N =\left(\expan {M_\e}N \circ T_\omega\right)^{-1} D_\mu f_{\e,\expan {\mu_\e} N}\circ \expan {K_\e}N$ is a trigonometric polynomial of degree $aN$.

The following proposition summarizes the computations presented above and assures that hypothesis \textbf{HTP1} and \textbf{HTP2} of the main Lemma \ref{main_theo} are satisfied for the dissipative standard map. 

\begin{proposition}\label{hyp-satisfy} For any $N\in\nat$, if $V(\t)$ in \eqref{dis-est-map-appendix} is a trigonometric polynomial of degree $a$, then 
$  \expanin {\tilde{E}_{\e,2}} N{2N}$ is a trigonometric polynomial of degree $3aN$, $\tilde{A}_{\e,2}^N$ is a trig polynomial of degree $aN$, and 
\begin{multline} \tilde{E}^N_{\Omega,\e}(\theta)\equiv D\expan {K_\e} N (\t+\omega)^\top J\circ \expan {K_\e} N(\t+\omega) D\expan {K_\e} N (\t+\omega) \\ 
			- D(f_{\e,\expan{\mu_\e} N}\circ \expan {K_\e} N(\t))^\top J\circ (f_{\e,\expan {\mu_\e} N}\circ \expan {K_\e} N(\t)) D(f_{\e,\expan{\mu_\e} N}\circ \expan {K_\e} N(\t))
\end{multline} is a trigonometric polynomial of degree $2aN$.
 \end{proposition} 

\begin{proof}
It is only left to prove the last claim. Note that $\tilde{E}^N_{\Omega,\e}(\theta)$ is the expression in coordinates of ${(\expan {K_\e} {N}\circ T_\omega)}^*  \Omega -(f_{\e,\expan\mu N}\circ \expan KN)^*\Omega$. Now, using the fact that $f_{\e,\mu}$ is conformally symplectic we have $(f_{\e,\expan{\mu_\e} N}\circ \expan {K_\e}N)^*\Omega = {\expan {K_\e}{N}}^* f_{\e,\expan{\mu_\e} N}^*\Omega = \l(\e) {\expan {K_\e}N}^* \Omega$, which means that, in coordinates \begin{multline}\tilde{E}^N_{\Omega,\e}(\theta,\e) = D\expan {K_\e} N (\t+\omega)^\top J\circ \expan {K_\e} N(\t+\omega) D\expan {K_\e} N (\t+\omega) \\  - \l(\e) D\expan {K_\e} N (\t)^\top J\circ \expan {K_\e} N(\t) D\expan {K_\e} N (\t)
\end{multline}
which is a polynomial of degree $2aN$ due to the fact that $J$ is a constant matrix and $$D\expan {K_\e}N(\t)= \begin{pmatrix}
1 +\sum_{n=0}^N u'_n(\t)\e^n \\ \sum_{n=0}^N(u'_n(\t) -u'_n(\t-\omega))\e^n  \end{pmatrix}$$ is a trigonometric polynomial of degree $aN$.
\end{proof}


\subsection{Uniqueness}
Note that for $\e=0$, $M_0=I$. Also note that the coefficients of the expansion \eqref{explic_K} are given by $$K_n(\t)=\begin{pmatrix} u_n(\t) \\ u_n(\t) -u_n(\t-\omega) \end{pmatrix}\qquad\mbox{for } n\geq 1 .$$ Therefore, the normalization condition $$\int_\torus \left[ M_0^{-1}K_n(\t)\right]_1d\t =0$$ in this case has the form $$\int_\torus u_n(\t)d\t =0,$$ which is satisfied by the construction of the $u_n's$. Thus, the expansion given in \eqref{explic_K} is the only one which satisfies the normalization condition.

\section*{Acknowledgements}

The authors would like to thank T. M-Seara, I. Baldom\'a, and V. Naudot for many suggestions and discussions.

\bibliography{gevreybib}
\bibliographystyle{alpha}

\Addresses
\end{document}